\documentclass[reqno]{amsart}
\usepackage{amsmath,amsthm,amssymb}

\usepackage{fullpage}
\usepackage{latexsym}
\usepackage{eucal}

\usepackage{srcltx}
\usepackage{eepic,epic,bez123,ebezier,xcolor,curves,calc,graphicx,latexpix,multiply,pgf}
\usepackage{hyperref}

\newtheorem{theorem}{Theorem}[section]
\newtheorem{lemma}{Lemma}[section]
\newtheorem{corollary}{Corollary}[section]

\theoremstyle{definition}

\def\cal{\mathcal}

\begin{document}
\title[The centrally symmetric $V$-states for active scalar equations. Two-dimensional Euler with cut-off]
{The centrally symmetric $V$-states for active scalar equations.
Two-dimensional Euler with cut-off}

\author{Sergey A. Denisov}


\address{
\begin{flushleft}
University of Wisconsin--Madison\\  Mathematics Department\\
480 Lincoln Dr., Madison, WI, 53706, USA\\
  denissov@math.wisc.edu
\end{flushleft}
  }
\maketitle

\maketitle
\begin{abstract}
We consider the family of active scalar equations on the plane and study the dynamics of two centrally symmetric patches. We focus on the two-dimensional Euler equation written in the vorticity form and consider its truncated version. For this model, a non-linear and non-local evolution equation is studied and a family of stationary solutions $\{y(x,\lambda)\}, x\in [-1,1], \lambda\in (0,\lambda_0)$ is found.  For these functions, we have $y(x,\lambda)\in C^\infty(-1,1)$ and
$\|y(x,\lambda)-|x|\|_{C[-1,1]}\to 0, \,\,\lambda\to 0$. The relation to the $V$-states observed numerically in \cite{over2, will} is discussed.

\end{abstract} \vspace{1cm}
\Large

\section{Introduction.}

In this paper, we study a certain class of the active scalar equations on the plane. Suppose we are given a function $D(z), z=(x,y)\in \mathbb{R}^2$ that satisfies the following properties: $D$ is radially symmetric, i.e., $D(z)=d(|z|)$, $d(r)$ is monotonically increasing and smooth on $(0,\infty)$. Consider the following transport equation
\begin{equation}\label{act-scal}
\dot \theta=\nabla \theta \cdot \left(\nabla^\perp A\theta+S\right),
\quad \theta(0,x,y)=\theta_0(x,y)
\end{equation}
where
\[
Af=\int_{\mathbb{R}^2} D(z-\xi)f(\xi)d\xi, \quad z,\xi\in
\mathbb{R}^2, \quad \nabla^{\perp}=(-\partial_y,\partial_x)
\]

 The symbol $S(t,z)$ will denote
the strain, i.e., an exterior velocity which is assumed to be incompressible and sufficiently regular. By choosing different $d(r)$ and $S(z,t)$, we can cover some important cases. For example, taking $d(r)=-r^{-1}$ and
$S(z,t)=0$ corresponds to the so-called surface-quasigeostrophic equation (SQG) for which only the local in time solvability is known for sufficiently smooth $\theta_0$ (see \cite{chae} for the recent development). If $d(r)=\log r$ and $S(z,t)=0$, one recovers the equation for vorticity for two-dimensional non-viscous Euler equation. In this situation, the existence of
global solution $\theta(t,x,y)$ has been known for a long time \cite{wol}.\smallskip

In this paper, we mostly focus on the Euler case, however, we will be digressing to more general situations later in the text. Let us assume $\theta_0\in L^\infty(\mathbb{R}^2)\cap L^1(\mathbb{R}^2)$. In that case, the existence of the global weak solution was established by Yudovich \cite{yud}. If $\theta_0=\chi_{\Omega_0}$ with some domain $\Omega_0: |\Omega_0|<\infty$, then one has the evolution of the ``patch" as $\theta(t,z)=\chi_{\Omega(t)}$ and $\Omega(t)$ is homeomorphic to $\Omega_0$.\smallskip

 We consider
the case when $\theta(0,z)=\chi_{\Omega_0}(z)+\chi_{-\Omega_0}(z)$
and $-\Omega_0=\{-z,\, z\in \Omega_0\}$. Assuming  $\Omega_0\cap
-\Omega_0=\varnothing$, one has
$\theta(t,z)=\chi_{\Omega(t)}(z)+\chi_{-\Omega(t)}(z)$, i.e., it
represents evolution of the centrally symmetric pair of patches (the
preservation of central symmetry is a simple feature of dynamics).
We also take $\Omega_0$ to be simply connected with smooth boundary, i.e. $\partial\Omega_0\in C^\infty$. Under these assumptions, we have $\partial\Omega(t)\in C^\infty$ for all time \cite{chem, berto}.

\smallskip

Two problems arise naturally in the study of this model. The first one addresses the following question. Let ${\mathfrak d}(t)={\rm dist}(\Omega(t), 0)$. Is it possible for ${\mathfrak d}(t)\to 0, \quad t\to\infty$? The Yudovich theory gives a lower bound: ${\mathfrak d}(t)>\exp(-\exp(Ct))$ and one can study its sharpness.   Naturally, the convergence of $\mathfrak{d}(t)$ to zero implies the merging of the patches as the configuration is centrally symmetric. In the recent paper \cite{den1}, the sharpness of the double exponential estimate was established (even up to a constant $C$) in the case when equation was considered with the strain $S$ which was assumed to be incompressible, odd, and Lipschitz-regular.  We have more evidence of the singularity formation in \eqref{act-scal}: in \cite{kis-sv}, the phenomenon of double exponential merging was proved for the Euler equation on the disc, however the presence of the boundary was used in a substantial way. For the related SQG model, there is a numerical evidence that two patches merge in finite time \cite{gang}.

 The second problem, intimately related to the first one, is  existence of the quasi-stationary states, i.e., the configurations of two centrally symmetric patches that rotate with constant angular velocity around the origin without changing the shape (the so-called $V$-states). For the 2D Euler, there is a numerical evidence (e.g., \cite{will,over2}
and references there) for the existence of the parametric
curve of these $V$-states: $V_\lambda \cup -V_\lambda$, i.e.,
$\Omega(t)=R_{\omega t}V_\lambda$, where $R_\theta$ denotes the
rotation around the origin by the angle $\theta$. Here, ${\rm dist}
(V_\lambda, -V_\lambda)=2\lambda$ and $\lambda\in [0,\lambda_0)$.
For $\lambda>0$, the boundary $\Gamma_\lambda=\partial V_\lambda$
seems to be smooth but the two patches form a sharp corner of $90$
degrees and touch at the origin when $\lambda=0$. Assuming existence
of the $V$--states in the contact position and their regularity away
from the origin, Overman \cite{over1} did a careful analysis around
the zero. In particular, he explained why the $90$ degrees is the
only possible nontrivial angle of the contact.\smallskip

The paper consists of eight sections and an Appendix. In the second section, the model with cut-off is introduced and the main result is stated. Section 3 gives some preliminaries and section 4 explains the strategy of the proof. In the sections 5, 6, and 7, we prove some auxiliary statements that are used in section 8 to complete the proofs of the main results. The Appendix has four lemmas we use in the main text.
\smallskip

{\it Notation used in the paper.}\quad  The symbol $\dot Lip[0,T]$ will indicate the following space $\dot
Lip [0,T]$$=\{f'\in L^\infty[0,T], f(0)=0\}$ equipped with the norm
\[
\|f\|_{\dot Lip[0,T]}=\sup_{x\in [0,T]}|f'(x)|
\]
We will use the following (non-standard)
notation
\[
\log^+ x=|\log x|+1, \quad x>0
\]
Let $\omega(x)$ be a smooth function such that $\omega(x)=1$ on
$|x|<1/2$, $\omega(x)=0$ on $|x|>1$ and $0\leq \omega(x)\leq 1$.
 For
a parameter $a>0$, we consider $\omega_a(x)=\omega(x/a)$ and
$\omega_a^{c}(x)=1-\omega_a(x)$. Given two positive functions $F_1$
and $F_2$, we write $F_1\lesssim F_2$ if there is a constant $C$
such that
\[
F_1<CF_2~,\quad C>0
\]
We write $F_1\sim F_2$ if
\[
F_1\lesssim F_2\lesssim F_1\,\,.
\]
The expression $``a\ll 1"$ will be a short-hand for ``there is a sufficiently small $a_0$ such that $0<a<a_0$". For the function $P(x)$, we write
$
\Delta_{x_1,x_2}P=P(x_1)-P(x_2)
$.
\bigskip

\section{The model with cut-off and the main result.}

 Consider $d(r)=\log r$ (2d Euler). If $\Omega_{sc}(0)$ is
a simply connected domain with smooth boundary, the evolution of
$\Gamma_{sc}(t)=\partial \Omega_{sc}(t)$ is governed by the
following integro-differential equation (see, e.g., \cite{mb},
formula (8.56); this is a corollary of $\nabla_z D(|z-\xi|)=-\nabla_{\xi} D(|z-\xi|)$ and the Green's formula):

\begin{equation}\label{global}
\dot z(t,\alpha)=C\int_{0}^{2\pi} z'(t,\beta)
\log|z(t,\beta)-z(t,\alpha)|d\beta, \quad \alpha\in[0,2\pi)
\end{equation}
where $C$ is an absolute constant and $z(\alpha,t)$ is anti-clockwise parameterization of the curve $\Gamma_{sc}(t)$. In particular, the right-hand side gives the velocity at any point on the boundary, $z(t,\alpha)$. If one has two simply connected domains $\Omega^{(1)}$ and $\Omega^{(2)}$ with vorticity equal to $1$ inside each of them, then the velocity at any point
$z_1\in \Gamma^{(1)}$ is given by
\[
C\left(\int_{0}^{2\pi} z_1'(t,\beta)
\log|z_1(t,\beta)-z_1(t,\alpha)|d\beta+\int_{0}^{2\pi} z_2'(t,\beta)
\log|z_2(t,\beta)-z_1(t,\alpha)|d\beta\right), \, \alpha\in[0,2\pi)
\]
Assume now that two patches are merging at the origin. Then, we can introduce the local chart in $\{(x,y): |x|<\delta, |y|<\delta\}$ and parameterize the corresponding contours by $(x,\mu_1(t,x))$ and $(x,\mu_2(t,x))$, see Figure~1.
Notice that the velocity at any point $(x,\mu_1(t,x))$ can now be written as
\[
C\left(\int_{-\delta}^{\delta}(1,\mu_1'(t,\xi))\log \left((x-\xi)^2+(\mu_1(t,x)-\mu_1(t,\xi))^2\right)d\xi-\right.
\]
\[
\left.\int_{-\delta}^{\delta}(1,\mu_2'(t,\xi))\log \left((x-\xi)^2+(\mu_1(t,x)-\mu_2(t,\xi))^2\right)d\xi\right)+R(x,\mu_1(t,x))
\]
where the negative sign in front of the second integral comes from the anti-clockwise parameterization for the contour $\Gamma^{(2)}$. Here $R$ is the velocity induced by the $\Bigl(\Gamma^{(1)}\cup\Gamma^{(2)}\Bigr)\cap \{|z|>\delta\}$. Clearly, $R$ is smooth inside $\{|x|<\delta/2, |y|<\delta/2\}$ and is equal to zero at the origin in the case when $\Omega^{(2)}=-\Omega^{(1)}$. Dropping this $R$ as a term negligible around the origin, we end up with the following expression for the velocity
\[
C\int_{-\delta}^{\delta} (A-B, \mu_1'(t,\xi)A-\mu_2'(t,\xi)B)d\xi
\]
where
\[
A=\log \left((x-\xi)^2+(\mu_1(t,x)-\mu_1(t,\xi))^2\right), \, B=\log \left((x-\xi)^2+(\mu_1(t,x)-\mu_2(t,\xi))^2\right)
\]
at every point $(x,\mu_1(t,x)), |x|<\delta$. Following, e.g., \cite{rodr}, we notice that the subtraction of any tangential vector from the velocity does not change the evolution of the contour. Thus, we subtract the vector-field
\[
C\int_{-\delta}^\delta (A-B)d\xi \cdot (1,\mu_1'(t,x))
\]
which gives a modified velocity
\[
u_{mod}(x,\mu_1(t,x))=C\left(0,\int_{-\delta}^\delta (\mu_1'(t,\xi)A-\mu_2'(t,\xi)B-\mu_1'(t,x)A+\mu_1'(t,x)B)d\xi\right)
\]
Notice that the first component of this vector is zero so we have an equation
\begin{equation}\label{lodka}
\dot\mu_1(t,x)=C\int_{-\delta}^\delta (\mu_1'(t,\xi)A-\mu_2'(t,\xi)B-\mu_1'(t,x)A+\mu_1'(t,x)B)d\xi, \quad |x|<\delta/2
\end{equation}
for the evolution of $\mu_1(t,x)$. The similar formula can be obtained for $\mu_2(t,x)$.
We will focus on the situation when the contours are centrally symmetric so $\mu_2(t,x)=-\mu_1(t,-x)$. That gives us the following equation for $\mu(t,x)=\mu_1(t,x)$:
\begin{equation}\label{coff}
\dot\mu(t,x)=C\int_{-\delta}^\delta (\mu'(t,x)-\mu'(t,\xi))\log \left(\frac{(x+\xi)^2+(\mu(t,x)+\mu(t,\xi))^2}{(x-\xi)^2+(\mu(t,x)-\mu(t,\xi))^2}\right)d\xi
\end{equation}
We allow this equation to hold on all of $[-\delta,\delta]$ and call it {\it an equation with cut-off}. The rescaling of time makes it possible to adjust the value of $C$. Moreover, the equation \eqref{coff} should be complemented by
\begin{equation}
\mu(0,x)=\mu_0(x), \quad \mu(t,\delta)=c(t)
\end{equation}
where $\mu_0(x)$ gives an initial position of the curve and $c(t)$ defines the {\it control} or the boundary value for this local transport equation.
\vspace{0.5cm}

\setlength{\unitlength}{0.150mm}
\begin{picture}(585,393)(15,-430)
  \allinethickness{0.254mm}\put(30,-260){\vector(1,0){695}} 
        \allinethickness{0.254mm}\put(375,-490){\vector(0,1){455}} 
        \allinethickness{0.254mm}\path(200,-95)(550,-95)(550,-445)(200,-445)(200,-95) 
        \allinethickness{0.254mm}\path(285,-170)(470,-170)(470,-355)(285,-355)(285,-170) 
        \spline(170,-165)(260,-250)(380,-210)(465,-220)(615,-165)

        \spline(170,-390)(240,-325)(335,-300)(400,-350)(500,-330)(615,-400)

        \put(595,-141){\shortstack{$\mu_1(t,x)$}} 
        \put(615,-371){\shortstack{$\mu_2(t,x)$}} 
        \put(555,-91){\shortstack{$(\delta,\delta)$}} 
        \put(430,-160){\shortstack{$(\delta/2,\delta/2)$}} 
        \put(715,-286){\shortstack{$X$}} 
        \put(345,-51){\shortstack{$Y$}} 
        \put(350,-283){\shortstack{$O$}} 
        
         \put(705,-490){\shortstack{Figure 1}}
\end{picture}
\vspace{1.5cm}

In the case of the general kernel $D$ in \eqref{act-scal}, the resulting equation can be reduced to the following form
\begin{equation}\label{general-sher}
\dot \mu(t,x)=C\int_{-\delta}^\delta
(\mu'(t,x)-\mu'(t,\xi))K(x,\xi)d\xi, \quad \mu(0,x)=\mu_0(x), \quad \mu(t,1)=c(t)
\end{equation}
where $C$ is a floating constant (can be adjusted by time scaling),
\[
K(x,\xi)=H((\mu(t,x)+\mu(t,\xi))^2+(x+\xi)^2)-H((\mu(t,x)-\mu(t,\xi))^2+(x-\xi)^2)
\]
 and $H(r)=d(\sqrt r)$. If the function $d(r)=\log r$ or is homogeneous, the scaling in $z$ allows one to assume that $\delta=1$ which we will do from now on (see Figure 2).

 \vspace{1cm}

\setlength{\unitlength}{0.150mm}
\begin{picture}(585,393)(15,-430)
        \allinethickness{0.254mm}\put(15,-235){\vector(1,0){585}} 
        \allinethickness{0.254mm}\put(290,-425){\vector(0,1){385}} 

        \spline(100,-45)(295,-230)(500,-85)
        \spline(100,-390)(295,-250)(500,-430)

        \put(560,-250){\makebox(0,0)[cc]{\shortstack{$X$}}} 
        \put(255,-51){\shortstack{$Y$}} 
        \put(265,-260){\shortstack{$O$}} 
        \allinethickness{0.254mm}\path(160,-100)(435,-100)(435,-375)(160,-375)(160,-100) 
        \put(450,-260){\shortstack{$1$}} 
        \put(460,-71){\shortstack{$\mu(t,x)$}} 
        \put(560,-430){\shortstack{Figure $2$}}
\end{picture}
\vspace{1cm}

We think that problem with the cut-off might serve as a good model to study the merging of the central pair. Indeed, the active scalar equations are nonlocal but it is believed that the singularity of the convolution kernel at $r=0$ is responsible for the strong instability (e.g., merging). That suggests a local version of the equation \eqref{global} might be interesting to study first. For this purpose, we take \eqref{general-sher} as a model.  The local in time solvability of \eqref{general-sher} is not known and will be addressed elsewhere.  However the ``raison d'\^{e}tre" is different and can be formulated as the

\bigskip

{\bf Problem 1.} {\it Is there a smooth solution to
\begin{equation}\label{ur-1}
\dot \mu(t,x)=\int_{-1}^1 (\mu'(t,x)-\mu'(t,\xi))K(x,\xi)d\xi, \quad
x\in (-1, 1)
\end{equation}
such that
\[
\mu(t,x)\to y_0(x)=|x|, \quad t\to \infty
\]
If so, what are the estimates (lower and upper) on
\[
\mathfrak{d}(t)=\min_{x\in [-1,1]} |\mu(t,x)-y_0(x)|?
\]
} The function $y_0(x)$ plays a very special role.
It is a stationary solution to (\ref{ur-1}) and it mimics locally the limiting case ($t=\infty$) considered in \cite{den1}. The advantage of the model \eqref{general-sher} is that we known this singular stationary configuration exactly and the problem~1 asks for the analysis of its dynamical stability. In particular, is it possible for ${\mathfrak d}(t)$ to converge to zero as double-exponential in the case when $d(r)=\log r$ (2d Euler)? In \cite{den1}, this question was answered affirmatively assuming that a regular strain is allowed (see Appendix in \cite{den1}). In other words, an approximate solution to \eqref{ur-1} was constructed and its self-similarity analysis was performed. The question remains though whether this strain can be dropped and this is the content of the problem 1.\smallskip

The problem 1 seems hard. The important step in understanding it is
to address the question of the stationary states for
(\ref{ur-1}).\bigskip

{\bf Problem 2.} {\it Find the family of  even positive functions
$y(x,\lambda)\in C^1[-1,1]$ such that
\begin{equation}\label{ur-2}
\int_{-1}^1 (y'(x,\lambda)-y'(\xi,\lambda))K(x,\xi)d\xi=0, \quad x\in [-1,1]
\end{equation}
and
\[
y(0,\lambda)=\lambda, \quad \lambda\in (0,\lambda_0); \quad
\|y(x,\lambda)-|x|\|_{C[-1,1]}\to 0, \quad \lambda\to 0
\]
}\smallskip \noindent Quite naturally, we will call these functions ``the even
$V$--states for the model with cut-off". Since the original problem
of the patch evolution is invariant with respect to rotations, we
expect the existence of other families of $V$--states that are not
necessarily even. \bigskip

The main result of this paper is the following theorem which
contains a solution to the problem 2 for the case of 2d Euler.
\begin{theorem}\label{main}
There is a  family of  even positive functions $y(x,\lambda)\in C^1[-1,1]$
such that
\begin{equation}\label{ura-2}
\int_{-1}^1 (y'(x,\lambda)-y'(\xi,\lambda))\log\left(
\frac{(x+\xi)^2+(y(x,\lambda)+y(\xi,\lambda))^2}{(x-\xi)^2+(y(x,\lambda)-y(\xi,\lambda))^2}
\right)d\xi=0, \quad x\in [-1,1]
\end{equation}
and
\[
y(0,\lambda)=\lambda, \quad \lambda\in (0,\lambda_0); \quad
\lim_{\lambda\to 0}\|y(x,\lambda)-|x|\|_{C[-1,1]}=0
\]
\end{theorem}
\bigskip

{\bf Remark.} This result does not immediately imply any progress on problem 1, however the developed technique might be
useful.

{\bf Remark.} The model with a cut-off we introduced is only a model, obviously. However, in the case of 2d Euler (or SQG) equation on the torus $\mathbb{T}^2=[-1,1]^2$, the analog of $y_0(x)$ is the following configuration: $\theta_s(x,y)={\rm sign}\, x\cdot {\rm sign}\, y$ which represents two patches that touch each other at the $90$ degrees angle. The method developed in this paper is likely to be directly applicable to the bifurcation analysis of this case which is NOT a model. We will address this issue elsewhere.

{\bf Remark.} The bifurcation analysis of the stationary states is a classical subject in the mechanics of fluids (see, e.g., \cite{sve1, zeng} for the recent developments). We, however, focus on the technically hard case when the {\it singular} stationary state is considered.\bigskip

\section{Preliminaries.}

The main result of this paper is solution to problem 2 in the case
of 2d Euler equation with a cut-off. We start with some
preliminary calculations for the general case as that will help us understand  problem 2 better.\smallskip

Assume that $y(x,\lambda)$ solves the problem 2. Since $y$ is even in $x$, we have
\begin{equation}\label{ur1}
y'(x,\lambda)\int_0^1 K_1(x,\xi)d\xi=\int_0^1
y'(\xi,\lambda)K_2(x,\xi)d\xi, \quad y(0,\lambda)=\lambda
\end{equation}
where
\[
K_1(x,\xi)=K(x,\xi)+K(x,-\xi)=
H((y(x)+y(\xi))^2+(x+\xi)^2)-H((y(x)-y(\xi))^2+(x-\xi)^2)
\]
\[
+H((y(x)+y(\xi))^2+(x-\xi)^2)-H((y(x)-y(\xi))^2+(x+\xi)^2)
\]
and
\[
K_2(x,\xi)=K(x,\xi)-K(x,-\xi)= H((y(x)+y(\xi))^2+(x+\xi)^2)-H((y(x)-y(\xi))^2+(x-\xi)^2)
\]
\[
-H((y(x)+y(\xi))^2+(x-\xi)^2)+H((y(x)-y(\xi))^2+(x+\xi)^2)
\]
We suppress the dependence of $y$ on $\lambda$ and just write
$y(x)$ here.

\subsection{The explicit solution for the model case}

Let us go back to the equation (\ref{act-scal}). Instead of taking
the singular kernels in the convolution, one can consider
the smooth bump $D(z)$. The ``typical" behavior around the origin
then would be, e.g.,
\[
D(z)=C+|z|^2+o(|z|^2), \quad |z|\to 0
\]
Keeping only the quadratic part, we get
\[
K(x,\xi)=4(y(x)y(\xi)+x\xi), \, K_1(x,\xi)=8y(x)y(\xi), \,
K_2(x,\xi)=8x\xi
\]
The equation \eqref{ur-1} takes the following form
\[
y'(x)y(x)\int_0^1 y(\xi)d\xi=x\int_0^1 \xi y'(\xi)d\xi
\]
which easily integrates to
\[
y(x)=\sqrt{\lambda^2+\frac{B}{A}x^2}
\]
where
\[
A=\int_0^1 y(x)dx, \, B=\int_0^1 x y'(x)dx
\]
We have the following compatibility equations
\[
\left\{
\begin{array}{cc}
B=\sqrt{\displaystyle \lambda^2+\frac BA}-A\\
A=\displaystyle \int_0^1\sqrt{\lambda^2+\frac BA x^2}dx
\end{array}
\right. \left\{
\begin{array}{cc}
B=\sqrt{\displaystyle \lambda^2+\frac BA}-A\\
\sqrt{AB}=\displaystyle \lambda^2\int_0^{\lambda^{-1}\sqrt{BA^{-1}}}
\sqrt{1+\xi^2}d\xi
\end{array}
\right.
\]
Introduce
\[
B/A=u, \quad AB=v
\]
Then
\[
v=\frac{u(\lambda^2+u)}{(u+1)^2}, \quad \sqrt
v=\lambda^2\int_0^{\lambda^{-1}\sqrt{u}}\sqrt{1+\xi^2}d\xi
\]
We assume that $\lambda\in (0,\lambda_0), \lambda_0\ll 1$ and
$|u-1|\ll 1$ and so $|v-1/4|\ll 1$. Therefore, if
\[
u=1+\alpha, \quad v=1/4+\beta, \quad \alpha,\beta\ll 1
\]
then
\[
\beta=\alpha/4+\lambda^2/4+O(\alpha^2+\lambda^2\alpha)
\]
and
\[
\alpha=2\beta-\lambda^2 \log\frac
1\lambda+O(\beta^2+\lambda^2\alpha)
\]
Thus,
\[
\beta=-0.5\lambda^2\log\frac
1\lambda+\frac{\lambda^2}{2}+O(\lambda^4\log^2\lambda),\quad
\alpha=-2\lambda^2\log\frac
1\lambda+\lambda^2+O(\lambda^4\log^2\lambda)
\]
This calculation shows that $V_\lambda$ exists and the asymptotics
in $\lambda\to 0$ can be easily established. Since
\[
y(x)=\sqrt{\lambda^2+(1+\alpha)x^2}, \quad \alpha<0
\]
the curve will intersect the line $y=x$ at the point
\[
x^*_\lambda=\frac{\lambda}{|\alpha|^{1/2}}= \left(2\log \frac
1\lambda\right)^{-1/2}(1+o(1))
\]
Now, let us address the question of self-similarity. Rescale
\[
\mu(\widehat x)=\lambda^{-1}y(\widehat
x\lambda)=\sqrt{1+(1+\alpha)\widehat x^2}, \quad |\widehat
x|<\lambda^{-1}
\]
This shows that
\[
\sup_{|\widehat x|<\lambda^{-1}}|\mu(\widehat x)-\sqrt{1+\widehat
x^2}|\to 0
\]
and so the self-similar behavior is global.

The model case we just considered is the situation in which the
interaction is substantially long-range and the self-similarity of
the stationary state is global. The curve that we have in the limit
is hyperbola. That seems like a common feature of many long-range
models and 2d Euler in particular as will be seen from the
subsequent analysis. However, for 2d Euler this self-similarity will
be proved only over $|x|<C\lambda$ with arbitrary fixed $C$. Notice
also that the analogous calculation is possible if the smooth
strains are imposed, e.g., a rotation.
\bigskip

\subsection{Properties of the kernels $K_1$ and $K_2$}

Below, we will write $K_{1(2)}(x,\xi,y)$ when we want to emphasize
the dependence of the kernel on the function $y$.

\begin{lemma}\label{lem1}
The following is true
\[
K_1(x,\xi,y)=4y(x)y(\xi)(H'(\eta_1)+H'(\eta_2)), \,
K_2(x,\xi,y)=4x\xi(H'(\alpha_1)+H'(\alpha_2))
\]
where
\[
\eta_1>(x+\xi)^2, \eta_2>(x-\xi)^2
\]
and
\[
\alpha_{1(2)}>(x-\xi)^2
\]
\end{lemma}
\begin{proof}
Apply the mean value theorem to the first and second terms in the
expression. This gives
\[
K_1=\Bigl( H((y(x)+y(\xi))^2+(x+\xi)^2)-H((y(x)-y(\xi))^2+(x+\xi)^2)
\Bigr)
\]
\[
+\Bigl(    H((y(x)+y(\xi))^2+(x-\xi)^2)-H((y(x)-y(\xi))^2+(x-\xi)^2)
\Bigr)
\]
and

\[
K_2=\Bigl( H((y(x)+y(\xi))^2+(x+\xi)^2)-H((y(x)+y(\xi))^2+(x-\xi)^2)
\Bigr)
\]
\[
+\Bigl(    H((y(x)-y(\xi))^2+(x+\xi)^2)-H((y(x)-y(\xi))^2+(x-\xi)^2)
\Bigr)
\]
\end{proof}

If $H=\log x$, we have the following representation
\[
K_1(x,\xi,y)=\log\left(
\frac{(x+\xi)^2+(y(x)+y(\xi))^2}{(x-\xi)^2+(y(x)-y(\xi))^2}\cdot
\frac{(x-\xi)^2+(y(x)+y(\xi))^2}{(x+\xi)^2+(y(x)-y(\xi))^2} \right)
\]
Then, assuming that $y(x)\geq 0$,
\[
 \frac{(x-\xi)^2+(y(x)+y(\xi))^2}{(x-\xi)^2+(y(x)-y(\xi))^2}
=1+\frac{4y(x)y(\xi)}{(x-\xi)^2+(y(x)-y(\xi))^2}\geq 1
\]
and
\[
 \frac{(x+\xi)^2+(y(x)+y(\xi))^2}{(x+\xi)^2+(y(x)-y(\xi))^2}
=1+\frac{4y(x)y(\xi)}{(x+\xi)^2+(y(x)-y(\xi))^2}\geq 1
\]
Therefore,  we have
\[
\log\left(1+\frac{4y(x)y(\xi)}{(x-\xi)^2+(y(x)-y(\xi))^2}\right)\leq
K_1\lesssim \log\left(1+\frac{4y(x)y(\xi)}{(x-\xi)^2}\right)
\]
provided that $y\geq 0$. Similarly, for $K_2$,
\[
\log\left(1+\frac{4x\xi}{(x-\xi)^2+(y(x)-y(\xi))^2}\right)\leq
K_2=\log\left(1+\frac{4x\xi}{(x-\xi)^2+(y(x)+y(\xi))^2}\right)
\]
\[
+\log\left(1+\frac{4x\xi}{(x-\xi)^2+(y(x)-y(\xi))^2}\right)
\]
\[
\lesssim \log\left( 1+\frac{4x\xi}{(x-\xi)^2}\right)
\]
and this holds for all $y$.\smallskip

The following lemma is trivial.
\begin{lemma}
Let $0\leq a\leq b\leq C$. Then, $b\sim a+b$ and
\[
\frac{1}{b-a}\int_a^b \frac{d\eta}{\eta}\sim b^{-1}\sim (a+b)^{-1},
\quad a>b/2
\]
and
\[
\frac{1}{a+b}\lesssim \frac{1}{b-a}\int_a^b
\frac{d\eta}{\eta}\lesssim  \frac{\log^+ a}{a+b}, \quad a<b/2
\]
\end{lemma}
\bigskip
Suppose that $y\in [0,C]$. Then, applying this lemma to $K_1$ with
\[
a=(y(x)-y(\xi))^2+(x+\xi)^2, \quad b=(y(x)+y(\xi))^2+(x+\xi)^2
\]
and then with
\[
a=(y(x)-y(\xi))^2+(x-\xi)^2, \quad b=(y(x)+y(\xi))^2+(x-\xi)^2
\]
gives
\begin{equation}\label{vza}
\frac{y(x)y(\xi)}{y^2(x)+y^2(\xi)+(x-\xi)^2}\lesssim K_1\lesssim
y(x)y(\xi)\frac{\log^+((x-\xi)^2+(y(x)-y(\xi))^2)}{y^2(x)+y^2(\xi)+(x-\xi)^2}
\end{equation}
For $K_2$, the same reasoning yields
\[
\frac{x\xi}{x^2+\xi^2+(y(x)-y(\xi))^2}\lesssim K_2\lesssim x\xi
\frac{\log^+((x-\xi)^2+(y(x)-y(\xi))^2)}{y^2(x)+y^2(\xi)+(x-\xi)^2}
\]\vspace{0.5cm}

\section{The implicit function theorem, the 2d Euler case.}
In this section, we will apply the scheme of the implicit function
theorem to the 2d Euler with cut-off which corresponds to $H(x)=\log
x$. However, we first notice that the problem allows the following
scaling.
\begin{lemma}
If $y(x)$ solves
\[
\int_{-1}^1 y'(x)\log
\left(\frac{(x+\xi)^2+(y(x)+y(\xi))^2}{(x-\xi)^2+(y(x)-y(\xi))^2}\right)d\xi=\int_{-1}^1
y'(\xi)\log \left(
\frac{(x+\xi)^2+(y(x)+y(\xi))^2}{(x-\xi)^2+(y(x)-y(\xi))^2}\right)d\xi
\]then $y_\alpha(x)=\alpha y(x/\alpha)$ solves
\begin{eqnarray*}
\int_{-\alpha}^\alpha y_\alpha'(x)\log
\left(\frac{(x+\xi)^2+(y_\alpha(x)+y_\alpha(\xi))^2}{(x-\xi)^2+(y_\alpha(x)-y_\alpha(\xi))^2}
\right)d\xi \hspace{2cm}  \\\hspace{2cm} =\int_{-\alpha}^\alpha
y'_\alpha(\xi)\log \left(
\frac{(x+\xi)^2+(y_\alpha(x)+y_\alpha(\xi))^2}{(x-\xi)^2+(y_\alpha(x)-y_\alpha(\xi))^2}\right)d\xi
\end{eqnarray*}
for every $\alpha>0$.
\end{lemma}
\begin{proof}
The proof is an immediate calculation.
\end{proof}

Consider $y_\lambda(x)$ and take
\[
\widehat y(\widehat x,\lambda)=\lambda^{-1}y(\widehat
x\lambda,\lambda), \quad |\widehat x|<\lambda^{-1}
\]
We will perform this scaling many times in the paper. It allows to
reduce the problem to the one on the larger interval $|\widehat
x|<\lambda^{-1}$ with the normalization $\widehat y(0,\lambda)=1$.

\bigskip

{\bf Remark 2.} The perturbative analysis done below will be carried out around the
hyperbola $\widehat y(\widehat x)=\sqrt{\widehat x^2+1}$, not
$|\widehat x|$. The explanation to that is the following. The model case suggests that  $\{\widehat y(\widehat
x,\lambda)\}$ might have  some limiting behavior as $\lambda\to 0$.
If so, can one guess the asymptotical curve? To this end, let us make very natural assumptions that
\[
\widehat y(\widehat x,\lambda)\to f(\widehat x), \quad \widehat
y'(\widehat x,\lambda)\to f'(\widehat x)
\]
 on every interval $\widehat x\in [-C,C]$
and that
\[
\widehat y(\widehat x,\lambda)=\widehat x(1+o(1)),\quad \widehat
y'(\widehat x,\lambda)=1+o(1), \quad |\widehat x|\gg 1
\]
uniformly in $\lambda\in (0,\lambda_0]$. For $|\widehat x|<C$,
\[
(f'(\widehat x)+o(1))\int_{0}^{1/\lambda}
\left[\log\left(1+\frac{4\widehat y(\widehat x,\lambda)\widehat
y(\widehat \xi,\lambda)}{(\widehat x-\widehat \xi)^2+(\widehat
y(\widehat x,\lambda)-\widehat y(\widehat
\xi,\lambda))^2}\right)\right.
\]
\[
\left. + \log\left(1+\frac{4\widehat y(\widehat x,\lambda)\widehat
y(\widehat \xi,\lambda)}{(\widehat x+\widehat \xi)^2+(\widehat
y(\widehat x,\lambda)-\widehat y(\widehat
\xi,\lambda))^2}\right)\right] d\widehat \xi
\]
\[
=\int_{0}^{1/\lambda}(1+o(1)) \left[\log\left(1+\frac{4\widehat
x\widehat \xi}{(\widehat x-\widehat \xi)^2+(\widehat y(\widehat
x,\lambda)-\widehat y(\widehat \xi,\lambda))^2}\right) \right.
\]
\[
\left.
 + \log\left(1+\frac{4\widehat x
\widehat \xi}{(\widehat x+\widehat \xi)^2+(\widehat y(\widehat
x,\lambda)-\widehat y(\widehat \xi,\lambda))^2}\right)\right]
d\widehat \xi
\]
For the l.h.s., the asymptotics of the integrand as $\widehat
\xi\to\infty$ is
\[
\frac{4\widehat y(\widehat x,\lambda)}{\widehat \xi}+o(\widehat
\xi^{-1})
\]
and for the r.h.s., it is
\[
\frac{4\widehat x}{\widehat \xi}+o(\widehat \xi^{-1})
\]
Here we work under assumption that $|\widehat x|<C$. Taking $\lambda\to 0$, we get
\[
(f'f-\widehat x)\log \Bigl(1/\lambda\Bigr)+o\Bigl(\log \Bigl(1/\lambda\Bigr)\Bigr)=0
\]
This leads to $f'f-\widehat x=0$ and (since $f(0)=1$)
\begin{equation}\label{predict}
f(\widehat x)=(\widehat x^2+1)^{1/2}
\end{equation}
This formula was obtained under strong assumptions so does not imply the
self-similarity per se. However, one can take
\[
\widetilde y(x,\lambda)=(x^2+\lambda^2)^{1/2}
\]
as an approximate solution. Plugging it into the equation, one can
represent the resulting correction as the strain. Similarly to
\cite{den1}, one can show that this strain satisfies the uniform
bound
\[
\sup_{|z|<1,\lambda\in (0,1)}\frac{|S(z,\lambda)|}{|z|}<C
\]
The novelty of the current paper is that we  construct the {\it exact}
solution and thus make \mbox{$S(z,\lambda)=0$}. It will also be
proved that the exact solutions converge to hyperbola in the scaling
limit but only locally, over $x\in I_\lambda$, where $|I_\lambda|\to
0, \lambda\to 0$.\bigskip

In the lemma below, we show that all possible solutions
$y(x,\lambda)$ have the following common feature.

\begin{lemma}
If $y(x)$ solves (\ref{ur1}), then there is $x^*\in (0,1)$ at which
$y(x^*)=x^*$. That is, the graph of $y(x)$ intersects the line
$y=x$.
\end{lemma}
\begin{proof}
Suppose instead that $y(x)>x$ for all $x\in (0,1)$. Then,
\[
\frac{4y(x)y(\xi)}{(x-\xi)^2+(y(x)-y(\xi))^2}>\frac{4x\xi}{(x-\xi)^2+(y(x)-y(\xi))^2}
\]
and
\[
\frac{4y(x)y(\xi)}{(x+\xi)^2+(y(x)-y(\xi))^2}>\frac{4x\xi}{(x-\xi)^2+(y(x)+y(\xi))^2}
\]
Therefore, $K_1(x,\xi)>K_2(x,\xi)>0$. Now, assume that
\[
\max_{x\in [0,1]} y'(x)=y'(x_1)
\]
Then,
\[
\int_0^1 y'(\xi)K_1(x_1,\xi)d\xi\leq y'(x_1)\int_0^1
K_1(x_1,\xi)d\xi=\int_0^1 y'(\xi)K_2(x_1,\xi)d\xi
\]
and this inequality  is strict unless $y'(x)={\rm const}$. This is impossible by, e.g., the smoothness assumption.
\end{proof}\bigskip

Now that we established what properties the solution $y(x,\lambda)$
needs to possess, we are ready to prove its existence.\smallskip

Consider small $\delta>0$ and the sets
\[
\Omega=\{f: \, \|f(x)-x\|_{\dot Lip[0,1]}\leq \delta\},  \quad
I=\{\lambda: \lambda\in (0,\lambda_0], \lambda_0\ll 1\}
\]
We will look for $y=\sqrt{\lambda^2+f^2(x)}$, where $(f,\lambda)\in
\Omega\times I$. Notice that $f(x)=\int_0^x f'(t)dt$ and
$\|f'-1\|_{L^\infty[0,1]}\ll 1$. Therefore,
\[
f(x)=x(1+O(\delta))
\]
In particular, $f(x)>0$ for $x>0$.\bigskip

Consider the functional (we specify the dependence of $K_{1(2)}$ on $y$
here)
\[
F(f,\lambda)= \frac{\displaystyle ff'\int_0^1
K_1(x,\tau,y)d\tau-\sqrt{\lambda^2+f^2(x)} \int_0^1
y'(\tau)K_2(x,\tau,y)d\tau}{x\sqrt{x^2+\lambda^2}\log^+(x^2+\lambda^2)}
\]
\[
{\rm where}\quad \quad y=\sqrt{\lambda^2+f^2(x)}
\]
which acts from $\Omega\times I$ to $L^\infty[0,1]$. Moreover,
$F(x,0)=0$.\smallskip

\bigskip

The equation \eqref{ur1} can be rewritten as
\[
F(f,\lambda)=0
\]
We will solve it in the following way (this is essentially the implicit function
theorem proof \cite{kh} but we prefer to give the argument for the sake of completeness). Write
\[
F(f,\lambda)=F(x,\lambda)+\Bigl( D_fF(x,\lambda)\Bigr) \psi+Q(\psi)
\]
where $\psi=f-x$ and this representation defines an operator $Q$.
That can be rewritten as
\begin{equation}\label{operat}
\psi=-\Bigl(D_fF(x,\lambda)\Bigr)^{-1}Q(\psi)+\psi_0(\lambda), \quad
\psi_0=-\Bigl(D_f F(x,\lambda)\Bigr)^{-1}F(x,\lambda)
\end{equation}
Next, we will show that this equation can be solved by contraction
mapping principle in $\cal{B}_\delta=\{\|\psi\|_{\dot Lip[0,1]}\leq
\delta\}, \,\delta\ll 1$. To this end, we only need to
prove:\bigskip
\begin{enumerate}

\item[(a)] Linear part:
\begin{equation}\label{aaa1}
\|\bigl(D_fF(x,\lambda)\bigr)^{-1}\|_{L^\infty[0,1], \dot Lip[0,1]}<\widehat{C}
\end{equation}
 if
$\lambda\in (0,\lambda_1)$ with $\lambda_1\ll 1$.\bigskip

\item[(b)] Frechet differentiability:
\begin{equation}\label{aaa2}
\|Q(\psi)\|_{ L^\infty[0,1]}=o(1)\|\psi\|_{\dot Lip[0,1]}
\end{equation}
 and
\begin{equation}\label{aaa3}
\|Q(\psi_2)-Q(\psi_1)\|_{ L^\infty[0,1]}=o(1)\|\psi_2-\psi_1\|_{\dot
Lip[0,1]}
\end{equation}
with $o(1)\to 0$ as $\delta\to 0$ uniformly in $\lambda\in (0,1)$ and
 $\psi, \psi_{1(2)}\in \cal{B}_\delta$.\bigskip

\item[(c)] Small initial data:
\begin{equation}\label{aaa4}
\|\psi_0(\lambda)\|_{\dot Lip[0,1]}<\delta/2
\end{equation}
where $\lambda\in [0,\lambda_0], \lambda_0<\lambda_1$.

\end{enumerate}\bigskip

We will first make $\lambda_1$ so small that (a) holds.
Then, we choose $\delta$ small enough to have $o(1)$ in (b) at most $(10\widehat{C})^{-1}$ uniformly in $\lambda\in (0,1)$. Finally, we take
 $\lambda_0$ so small that (c) holds. This
will ensure  existence and  uniqueness of solution in the
 complete metric space $\cal{B}_\delta$. Then, it will be easy to bootstrap its regularity from $Lip[-1,1]$ to $C^1[-1,1]$. The continuous
dependence on $\lambda$ and
\[
\|y(x,\lambda)- x\|_{C[0,1]}\to 0, \quad \lambda\to 0
\]
will follow from the proof.

\bigskip

\section{ The analysis of  Gateaux derivative for $H(x)=\log
x$.}\bigskip

Taking $f_t=f+tu, u\in \dot Lip[0,1]$, plugging it into $F$, and
computing the derivative in $t$ at $t=0$ with positive $x$ fixed,
results in
\begin{equation}\label{proizvol}
(D_fF(f,\lambda))u=\frac{1}{x\sqrt{x^2+\lambda^2}\log^+(x^2+\lambda^2)}\left(
I_1+\ldots +I_6\right)
\end{equation}
We have
\[
I_1=\left(f'\int_0^1 K_1(x,\tau,y)d\tau\right)u, \quad
y=\sqrt{\lambda^2+f^2}
\]
\[
I_2=\left(  f\int_0^1 K_1(x,\tau,y)d\tau \right)u'
\]
\[
I_3=ff'\int_0^1 \delta K_1(x,\tau,y)d\tau
\]
where
\[
\delta K_1=\frac{2(y(x)+y(\xi))(\delta y(x)+\delta
y(\xi))}{(x+\xi)^2+(y(x)+y(\xi))^2} - \frac{2(y(x)-y(\xi))(\delta
y(x)-\delta y(\xi))}{(x-\xi)^2+(y(x)-y(\xi))^2}
\]
\[
+\frac{2(y(x)+y(\xi))(\delta  y(x)+\delta y
(\xi))}{(x-\xi)^2+(y(x)+y(\xi))^2} -\frac{2(y(x)-y(\xi))(\delta y
(x)-\delta y  (\xi))}{(x+\xi)^2+(y(x)-y(\xi))^2}
\]
and
\[
\delta y=\frac{fu}{\sqrt{\lambda^2+f^2}}
\]
\[
I_4=-\left(\frac{f}{\sqrt{\lambda^2+f^2}}\int_0^1
y'K_2(x,\tau,y)d\tau\right) u
\]
\[
I_5=-\sqrt{\lambda^2+f^2}\int_0^1 \delta y' K_2(x,\tau,y)d\tau
\]
where
\[
\delta
y'=\frac{f'}{\sqrt{\lambda^2+f^2}}u+\frac{f}{\sqrt{\lambda^2+f^2}}u'-\frac{f^2f'u}{(\lambda^2+f^2)^{3/2}}
\]
\[
I_6=-\sqrt{\lambda^2+f^2}\int_0^1 y'(\tau)\delta K_2(x,\tau,y)d\tau
\]
where
\[
\delta K_2=\frac{2(y(x)+y(\xi))(\delta y(x)+\delta
y(\xi))}{(x+\xi)^2+(y(x)+y(\xi))^2} - \frac{2(y(x)-y(\xi))(\delta
y(x)-\delta y(\xi))}{(x-\xi)^2+(y(x)-y(\xi))^2}
\]
\[
-\frac{2(y(x)+y(\xi))(\delta y(x)+\delta
y(\xi))}{(x-\xi)^2+(y(x)+y(\xi))^2} +\frac{2(y(x)-y(\xi))(\delta
y(x)-\delta y(\xi))}{(x+\xi)^2+(y(x)-y(\xi))^2}
\]\bigskip

\subsection{ The derivative at $f(x)=x$}\bigskip

Define $L_\lambda=(D_fF)(x,\lambda)$. If $f=x$ in the previous
section, then
\[
L_\lambda
u=\frac{1}{x\sqrt{x^2+\lambda^2}\log^+(x^2+\lambda^2)}\left(
\widehat I_{1,\lambda}+\ldots +\widehat I_{6,\lambda}\right)
\]
We again have
\[
\widehat I_{1,\lambda}=\left(\int_0^1
K_1(x,\tau,y_\lambda)d\tau\right)u
\]
 with
\[
y_\lambda(x)=\sqrt{\lambda^2+x^2}
\]

\[
\widehat I_{2,\lambda}=x\left(  \int_0^1  K_1(x,\tau,y_\lambda)d\tau
\right)u'
\]
\[
\widehat I_{3,\lambda}=x\int_0^1 \delta K_1(x,\tau,y_\lambda)d\tau
\]
where
\[
\delta K_1=\frac{2(y_\lambda(x)+y_\lambda(\xi))(\delta
y_\lambda(x)+\delta
y_\lambda(\xi))}{(x+\xi)^2+(y_\lambda(x)+y_\lambda(\xi))^2} -
\frac{2(y_\lambda(x)-y_\lambda(\xi))(\delta y_\lambda(x)-\delta
y_\lambda(\xi))}{(x-\xi)^2+(y_\lambda(x)-y_\lambda(\xi))^2}
\]
\[
+\frac{2(y_\lambda(x)+y_\lambda(\xi))(\delta  y_\lambda(x)+\delta
y_\lambda (\xi))}{(x-\xi)^2+(y_\lambda(x)+y_\lambda(\xi))^2}
-\frac{2(y_\lambda(x)-y_\lambda(\xi))(\delta y_\lambda (x)-\delta y_\lambda
(\xi))}{(x+\xi)^2+(y_\lambda(x)-y_\lambda(\xi))^2}
\]
and
\[
\delta y_\lambda=\frac{x}{\sqrt{\lambda^2+x^2}}u
\]
\[
\widehat I_{4,\lambda}=-\left(\frac{x}{\sqrt{\lambda^2+x^2}}\int_0^1
y_\lambda'K_2(x,\tau,y_\lambda)d\tau\right) u
\]
\[
\widehat I_{5,\lambda}=-\sqrt{\lambda^2+x^2}\int_0^1 \delta
y_\lambda' K_2(x,\tau,y_\lambda)d\tau
\]
where
\[
\delta
y'_\lambda=\frac{1}{\sqrt{\lambda^2+x^2}}u+\frac{x}{\sqrt{\lambda^2+x^2}}u'
-\frac{x^2}{(\lambda^2+x^2)^{3/2}}u
\]
\[
\widehat I_{6,\lambda}=-\sqrt{\lambda^2+x^2}\int_0^1
y'_\lambda(\tau)\delta K_2(x,\tau,y_\lambda)d\tau
\]
and
\[
\delta K_2=\frac{2(y_\lambda(x)+y_\lambda(\xi))(\delta
y_\lambda(x)+\delta
y_\lambda(\xi))}{(x+\xi)^2+(y_\lambda(x)+y_\lambda(\xi))^2} -
\frac{2(y_\lambda(x)-y_\lambda(\xi))(\delta y_\lambda(x)-\delta
y_\lambda(\xi))}{(x-\xi)^2+(y_\lambda(x)-y_\lambda(\xi))^2}
\]
\[
-\frac{2(y_\lambda(x)+y_\lambda(\xi))(\delta y_\lambda(x)+\delta
y_\lambda(\xi))}{(x-\xi)^2+(y_\lambda(x)+y_\lambda(\xi))^2}
+\frac{2(y_\lambda(x)-y_\lambda(\xi))(\delta y_\lambda(x)-\delta
y_\lambda(\xi))}{(x+\xi)^2+(y_\lambda(x)-y_\lambda(\xi))^2}
\]

\bigskip

\subsection{ The operator $L_\lambda$}\bigskip

For $L_\lambda$, we have the following formula
\[
L_\lambda=A_1u'+A_2u+\int_0^1 D_1(x,\xi,\lambda)u(\xi)d\xi+\int_0^1
D_2(x,\xi,\lambda)u'(\xi)d\xi
\]
The equation
\[
L_\lambda u=g
\]
can be rewritten as
\begin{equation}\label{for1}
A_1(x,\lambda)u'+A_2(x,\lambda) u+\int_0^1
M(x,\xi,\lambda)u'(\xi)d\xi=g
\end{equation}
if one assumes $u(0)=0$ and
\[
M(x,\xi,\lambda)=D_2(x,\xi,\lambda)+\int_{\xi}^1
D_1(x,\tau,\lambda)d\tau
\]
In the calculation above, we used
\[
\lim_{x\to 0}\left(u(x)\int_0^1 D_1(x,\tau,\lambda)d\tau\right)=0
\]
This equality follows from the estimate $|u(x)|\lesssim x$ and from the
analysis of \[\int_0^1D_1(x,\tau,\lambda)d\tau\] when $x\to 0$ (see
\eqref{de1} below).

Let us introduce the integral operator $\cal M_\lambda$ with the
kernel $M(x,\tau,\lambda)$, e.g.,
\[
{\cal M}_\lambda f=\int_0^1 M(x,\tau,\lambda)f(\tau)d\tau
\]
For the coefficients, we have
\[
A_1=\frac{\displaystyle \int_0^1
K_1(x,\tau,y_\lambda)d\tau}{\sqrt{x^2+\lambda^2}\log^+(x^2+\lambda^2)}
\]
The expression for $A_2$ is more complicated,
\begin{eqnarray}\label{a2}
A_2=\frac{1}{x\sqrt{x^2+\lambda^2}\log^+(x^2+\lambda^2)}\left(
\int_0^1 K_1(x,\tau, y_\lambda)d\tau- \right. \hspace{3cm}
\\\nonumber
\left. \frac{x}{\sqrt{\lambda^2+x^2}}\int_0^1
y'_\lambda(\tau)K_2(x,\tau,y_\lambda ) d\tau+B_2 \right)
\end{eqnarray}
where
\[
B_2=\frac{2x}{\sqrt{x^2+\lambda^2}}\int_0^1 \left(x-
\frac{\xi\sqrt{\lambda^2+x^2}}{\sqrt{\lambda^2+\xi^2}}\right) \left(
\frac{y_\lambda(x)+y_\lambda(\xi)}{(x+\xi)^2+(y_\lambda(x)+y_\lambda(\xi))^2}
\right.\]

\[\left.
-\frac{y_\lambda(x)-y_\lambda(\xi)}{(x-\xi)^2+(y_\lambda(x)-y_\lambda(\xi))^2}
\right)d\xi
\]
\[
+\frac{2x}{\sqrt{x^2+\lambda^2}}\int_0^1 \left(x+
\frac{\xi\sqrt{\lambda^2+x^2}}{\sqrt{\lambda^2+\xi^2}}\right) \left(
\frac{y_\lambda(x)+y_\lambda(\xi)}{(x-\xi)^2+(y_\lambda(x)+y_\lambda(\xi))^2}
\right.
\]
\[
\left.
-\frac{y_\lambda(x)-y_\lambda(\xi)}{(x+\xi)^2+(y_\lambda(x)-y_\lambda(\xi))^2}
\right)d\xi
\]

For $D_{1(2)}$, one has

\[
D_2(x,\xi,\lambda)=-\frac{1}{x\log^+(x^2+\lambda^2)}K_2(x,\xi,y_\lambda)\frac{\xi}{\sqrt{\lambda^2+\xi^2}}
\]
and
\[
D_1(x,\xi,\lambda)=\frac{1}{x\sqrt{\lambda^2+x^2}\log^+(x^2+\lambda^2)}\left[
\frac{2x\xi}{\sqrt{\lambda^2+\xi^2}}\left(
\frac{y_\lambda(x)+y_\lambda(\xi)}{(x+\xi)^2+(y_\lambda(x)+y_\lambda(\xi))^2}
\right. \right.
\]
\[
\left.+\frac{y_\lambda(x)-y_\lambda(\xi)}{(x-\xi)^2+(y_\lambda(x)-y_\lambda(\xi))^2}+\frac{y_\lambda(x)+y_\lambda(\xi)}{(x-\xi)^2+(y_\lambda(x)+y_\lambda(\xi))^2}+\frac{y_\lambda(x)-y_\lambda(\xi)}{(x+\xi)^2+(y_\lambda(x)-y_\lambda(\xi))^2}\right)
\]
\[
-\frac{2\xi^2\sqrt{\lambda^2+x^2}}{\xi^2+\lambda^2} \left(
\frac{y_\lambda(x)+y_\lambda(\xi)}{(x+\xi)^2+(y_\lambda(x)+y_\lambda(\xi))^2}
\right.
\]
\[
\left.+\frac{y_\lambda(x)-y_\lambda(\xi)}{(x-\xi)^2+(y_\lambda(x)-y_\lambda(\xi))^2}-\frac{y_\lambda(x)+y_\lambda(\xi)}{(x-\xi)^2+(y_\lambda(x)+y_\lambda(\xi))^2}-\frac{y_\lambda(x)-y_\lambda(\xi)}{(x+\xi)^2+(y_\lambda(x)-y_\lambda(\xi))^2}\right)
\]
\[
\left.-\frac{\lambda^2\sqrt{\lambda^2+x^2}}{(\lambda^2+\xi^2)^{3/2}}K_2(x,\xi,y_\lambda)\right]
\]

In this section, we will obtain estimates/asymptotics of all four
terms in the case when $\lambda\to 0$. It will be trivial to do that
away from $0$: e.g., for every $\delta>0$ both $A_{1(2)}(\lambda)\to
A_{1(2)}(0)$ uniformly over $x\in [\delta,1]$. The behavior around
$0$ is delicate and will require more careful treatment.\bigskip

We start with the following calculation that will simplify the
expressions above.

We write
\begin{equation}\label{miracle1}
\sqrt{\widehat x^2+1}-\sqrt{\widehat \xi^2+1}=(\widehat x-\widehat
\xi)r_1(x,\xi)
\end{equation}
where
\begin{equation}\label{miren}
r_1=\frac{\widehat x+\widehat \xi}{\sqrt{\widehat
x^2+1}+\sqrt{\widehat \xi^2+1}}=1+O\Bigl(\frac{1}{\widehat x\widehat
\xi}\Bigr), \quad {\rm if}\quad \widehat x,\widehat\xi\gg 1
\end{equation}
Similarly,
\begin{equation}\label{miracle2}
\sqrt{\widehat x^2+1}+\sqrt{\widehat \xi^2+1}=(\widehat x+\widehat
\xi)r^{-1}_1, \quad r_1^{-1}=\frac{\sqrt{\widehat x^2+1}+\sqrt{\widehat \xi^2+1}}{\widehat x+\widehat
\xi}
\end{equation}
Thus, we have for $K_2$
\begin{equation}\label{wow}
\frac{(\widehat x+\widehat \xi)^2+(\sqrt{\widehat
x^2+1}+\sqrt{\widehat \xi^2+1})^2}{(\widehat x-\widehat
\xi)^2+(\sqrt{\widehat x^2+1}-\sqrt{\widehat \xi^2+1})^2}\cdot
\frac{(\widehat x+\widehat \xi)^2+(\sqrt{\widehat
x^2+1}-\sqrt{\widehat \xi^2+1})^2}{(\widehat x-\widehat
\xi)^2+(\sqrt{\widehat x^2+1}+\sqrt{\widehat \xi^2+1})^2}
\end{equation}
\[
=\frac{(\widehat x+\widehat\xi)^2(1+r_1^{-2})}{(\widehat
x-\widehat\xi)^2(1+r_1^2)}\cdot \frac{(\widehat x+\widehat\xi)^2+(\widehat
x-\widehat\xi)^2r_1^2}{(\widehat x-\widehat\xi)^2+(\widehat
x+\widehat\xi)^2r_1^{-2}}=\frac{(\widehat
x+\widehat\xi)^2}{(\widehat x-\widehat\xi)^2}
\]
after the  cancelation.

Similarly, for $K_1$
\begin{equation}\label{wow1}
\frac{(\widehat x+\widehat \xi)^2+(\sqrt{\widehat
x^2+1}+\sqrt{\widehat \xi^2+1})^2}{(\widehat x-\widehat
\xi)^2+(\sqrt{\widehat x^2+1}-\sqrt{\widehat \xi^2+1})^2}\cdot
\frac{(\widehat x-\widehat \xi)^2+(\sqrt{\widehat
x^2+1}+\sqrt{\widehat \xi^2+1})^2}{(\widehat x+\widehat
\xi)^2+(\sqrt{\widehat x^2+1}-\sqrt{\widehat \xi^2+1})^2}
\end{equation}
\[=
\frac{(\widehat x+\widehat\xi)^2}{(\widehat
x-\widehat\xi)^2}r_1^{-4}
\]
Therefore, we have
\begin{equation}\label{miracle-1}
K_2(x,\tau,y_\lambda)=K_2(x,\tau,y_0)
\end{equation}
and
\begin{equation}\label{miracle-2}
K_1(x,\tau,y_\lambda)=K_1(x,\tau,y_0)-4\log r_1
\end{equation}
\bigskip

Now, we are ready for the analysis of the asymptotics for the
coefficients in $L_\lambda$.\bigskip

{\bf 1. The coefficient $A_1$.}\hspace{0.5cm}\bigskip

Consider $A_1(x,0)$ first. We have
\begin{equation}\label{a10}
A_1(x,0)=\frac{1}{x\log^+(x^2)}\int_0^1
\log\left(\frac{x+\xi}{x-\xi}\right)^2d\xi \hspace{4cm}
\end{equation}
\[
=\frac{1}{\log^+(x^2)}\int_0^{1/x}
\log\left(\frac{1+u}{1-u}\right)^2du=2+o(1), \quad x\to 0
\]
and it is smooth in $(0,1)$. At the point $x=0$, we define $A_1(0,0)=2$, i.e., by its right limit.

\begin{lemma}
We have
\begin{equation}\label{a1-1}
\lim_{\lambda\to 0}\|A_1(x,\lambda)-A_1(x,0)\|_{C[0,1]}=0
\end{equation}
\end{lemma}
\begin{proof}

 If $x=\lambda \widehat
x$, then
\[
\int_0^1 K_1(x,\tau,y_\lambda)d\tau=
\]
\[
\lambda \int_0^{1/\lambda} \log\left( \frac{(\widehat x+\widehat
\xi)^2+(\sqrt{\widehat x^2+1}+\sqrt{\widehat \xi^2+1})^2}{(\widehat
x-\widehat \xi)^2+(\sqrt{\widehat x^2+1}-\sqrt{\widehat
\xi^2+1})^2}\cdot  \frac{(\widehat x-\widehat \xi)^2+(\sqrt{\widehat
x^2+1}+\sqrt{\widehat \xi^2+1})^2}{(\widehat x+\widehat
\xi)^2+(\sqrt{\widehat x^2+1}-\sqrt{\widehat \xi^2+1})^2}
\right)d\widehat \xi
\]
 Many estimates done below will be based on the following standard argument that we explain now in detail.

We have several regimes:

(1). $\widehat x\in [0, 1]$. Notice that integration over any fixed interval $\widehat\xi\in [0,C]$ gives a contribution $O(\lambda)$, so we only need to control large $\widehat\xi$. Using \eqref{wow1}, one gets the following asymptotics for the expression under the logarithm
\[
\frac{(\widehat x+\widehat\xi)^2}{(\widehat
x-\widehat\xi)^2}r_1^{-4}=\left(1+\frac{4\widehat x}{\widehat\xi}+O(\widehat\xi^{-2})\right)
\left(
1+4\frac{\sqrt{\widehat x^2+1}-\widehat x}{\widehat \xi}+O(\widehat\xi^{-2})
\right), \quad\widehat\xi\to \infty
\]
Then, using the Taylor expansion for the logarithm, we get
\[
\lambda \int_0^{1/\lambda} \log\left( \frac{(\widehat x+\widehat
\xi)^2+(\sqrt{\widehat x^2+1}+\sqrt{\widehat \xi^2+1})^2}{(\widehat
x-\widehat \xi)^2+(\sqrt{\widehat x^2+1}-\sqrt{\widehat
\xi^2+1})^2}\cdot  \frac{(\widehat x-\widehat \xi)^2+(\sqrt{\widehat
x^2+1}+\sqrt{\widehat \xi^2+1})^2}{(\widehat x+\widehat
\xi)^2+(\sqrt{\widehat x^2+1}-\sqrt{\widehat \xi^2+1})^2}
\right)d\widehat \xi
\]
\begin{equation}\label{odnushka}
=4\lambda\sqrt{\widehat x^2+1}\log
(1/\lambda)+O(\lambda)=4\sqrt{x^2+\lambda^2}\log
(1/\lambda)+O(\lambda)=
\end{equation}
\[
\hspace{3cm}=2\sqrt{x^2+\lambda^2}\log^+(x^2+\lambda^2)+O(\lambda)
\]
\smallskip
Given any fixed $\delta\in (0,1)$, we have two cases.\smallskip

(2). Take $x\in (\delta, 1] $. We trivially get
\begin{equation}\label{semech}
\lim_{\lambda\to 0}\max_{x\in [\delta,1]}\left|\int_0^1 K_1(x,\tau,y_\lambda)d\tau- \int_0^1
K_1(x,\tau,y_0)d\tau\right|=0
\end{equation}

(3). Let $x\in (\lambda,\delta]$. We substitute \eqref{miracle2} to \eqref{wow1} and get
\[
\int_0^1 K_1(x,\tau,y_\lambda)d\tau=
2\lambda\int_0^{1/\lambda} \log\left|\frac{\widehat x+\widehat \xi}{\widehat\xi-\widehat x}\right|d\widehat\xi+4\lambda\int_0^{1/\lambda}\log\left(1+\frac{\sqrt{\widehat x^2+1}-\widehat x}{\widehat x+\widehat\xi}+O(\widehat\xi^{-2})\right)d\widehat\xi
\]
\[
=2x\int_0^{1/x}\log\left|\frac{1+t}{1-t}\right|dt+4\lambda(\sqrt{1+\widehat x^2}-\widehat x)\int_0^{1/x} \frac{1}{1+t}dt+O(\lambda)
\]
\[
=4x\log(1/x)+O(x)+4\lambda(\sqrt{1+\widehat x^2}-\widehat
x)\log(1/x)+O(\lambda)
=4\log(1/x)(x+\lambda\sqrt{1+\widehat x^2}-\lambda\widehat x)+O(x)
\]
\begin{equation}\label{treshka}
=2\sqrt{x^2+\lambda^2}\log^+(x^2+\lambda^2)+O(x)
\end{equation}
The bounds above imply
\begin{equation}\label{hevo}
\lim_{\lambda\to 0}\left\|\frac{\displaystyle \int_0^1
K_1(x,\tau,y_\lambda)d\tau}{\sqrt{x^2+\lambda^2}\log^+(x^2+\lambda^2)}-
\frac{\displaystyle \int_0^1 K_1(x,\tau,y_0)d\tau}{x\log^+(x^2)}\right\|_{L^\infty[0,1]}=0
\end{equation}
Indeed, given any $\epsilon>0$, we use \eqref{a10},\eqref{odnushka}, and \eqref{treshka} to get
\[
\left\|\frac{\displaystyle \int_0^1
K_1(x,\tau,y_\lambda)d\tau}{\sqrt{x^2+\lambda^2}\log^+(x^2+\lambda^2)}-
\frac{\displaystyle \int_0^1 K_1(x,\tau,y_0)d\tau}{x\log^+(x^2)}\right\|_{L^\infty[0,\delta]}\lesssim \frac{1}{\log^+(\delta^2+\lambda^2)}<\epsilon/2
\]
for $\delta<\delta(\epsilon)$ and $\lambda<\delta(\epsilon)$. For fixed $\delta<\delta(\epsilon)$, we have
\[
\left\|\frac{\displaystyle \int_0^1
K_1(x,\tau,y_\lambda)d\tau}{\sqrt{x^2+\lambda^2}\log^+(x^2+\lambda^2)}-
\frac{\displaystyle \int_0^1 K_1(x,\tau,y_0)d\tau}{x\log^+(x^2)}\right\|_{L^\infty[\delta,1]}\leq \epsilon/2,
\]
as long as $\lambda<\lambda(\epsilon)$ (by \eqref{semech}). This yields \eqref{hevo}.
  \smallskip
\end{proof}

 Later, we will need the following result
 \begin{lemma} \label{auxi} Suppose $\|g(x)-x\|_{\dot Lip[0,1]}\leq \delta\ll
 1$. Then,
\[
\left|\frac{\displaystyle \int_0^1
K_1(x,\tau,\sqrt{\lambda^2+g^2(\tau)})d\tau}{\sqrt{x^2+\lambda^2}\log^+(x^2+\lambda^2)}\right|\lesssim
1
\]
uniformly in $x\in [0,1]$, $\lambda\in (0,1]$, and $g$.
 \end{lemma}
Its proof repeats the argument in the previous lemma (see also the
proof of lemma \ref{hard} below to check how the problem can be
reduced to the homogeneous one for which the scaling can be easily
performed to get the desired bound). This result can also be
obtained by comparing to the case $g=x$ and using the stability
estimates established in lemma \ref{bbb1} below.

\bigskip

{\bf 2. The coefficient $A_2$.}\quad\quad

\begin{lemma}
 For every fixed $\delta>0$, we have
\begin{equation}\label{a2-1}
A_2(x,\lambda)\to A_2(x,0)=\frac{2\log(x^{-2}+1)}{x\log^+(x^2)},
\quad \lambda\to 0
\end{equation}
uniformly over $x\in [\delta,1]$. Moreover, we have an estimate
\begin{equation}\label{a2-2}
 A_2(x,\lambda)\sim \frac{1}{x}
\end{equation}
which holds uniformly in $x\in (0,1]$ and $\lambda\in (0,1]$.

\end{lemma}
\begin{proof}
The expression for $A_2(x,0)$ is easy to compute and the first part
of the lemma is immediate. The formula for $A_2(x,\lambda)$ contains three terms. The first one
involves $K_1$ and its asymptotics was established before. Consider
the second term. By \eqref{miracle1}, we get
\[
\int_0^1 y_\lambda'K_2(x,\tau,y_\lambda)d\tau= \lambda
\int_0^{1/\lambda} \frac{\widehat \xi}{\sqrt{\widehat
\xi^2+1}}\log\left(\frac{\widehat x+\widehat\xi}{\widehat
x-\widehat\xi}\right)^2d\widehat\xi
\]
The similar analysis yields:

(1). Uniformly in $x\in (\delta,1]$, we get
\begin{equation}
\label{shod1} \int_0^1 y_\lambda'K_2(x,\tau,y_\lambda)d\tau\to
\int_0^1 K_2(x,\tau,y_0)d\tau, \quad {\rm as}\quad \lambda\to 0
\end{equation}

(2). If $\widehat x\in [0,1]$, then we can split the integral into
two. The first one is
\[
\int_0^{1} \frac{\widehat \xi}{\sqrt{\widehat
\xi^2+1}}\log\left(\frac{\widehat x+\widehat\xi}{\widehat
x-\widehat\xi}\right)^2d\widehat\xi
\]
We have
\[
\int_0^1 \widehat \xi \log\Bigr(1+\frac{2\widehat x\widehat
\xi}{(\widehat x-\widehat\xi)^2}\Bigl)d\widehat \xi=\widehat x^2
\int_{0}^{\widehat x^{-1}}
t\log\Bigl(1+\frac{2t}{(1-t)^2}\Bigr)dt\sim \widehat x
\]
So, the integration over $[0,1]$ amounts to $O(x)$ after
multiplication by $\lambda$.

For the integral over $[1,\lambda^{-1}]$, we get
\[
\int_1^{1/\lambda} \frac{\widehat \xi}{\sqrt{\widehat
\xi^2+1}}\log\left(\frac{\widehat x+\widehat\xi}{\widehat
x-\widehat\xi}\right)^2d\widehat\xi =4\widehat
x\log(1/\lambda)+O(\widehat x)
\]
Multiplication by $\lambda$ yields
\[
\int_0^1 y_\lambda'K_2(x,\tau,y_\lambda)d\tau
=x\Bigl(4\log(1/\lambda)+O(1)\Bigr)
\]

(3). If $x\in (\lambda, \delta)$, then the integral over $[0,1]$ can
be handled as before and its contribution is at most $\widehat
x^{-1}$. The integral over $[1,1/\lambda]$ gives
\[
\int_1^{1/\lambda}(1+O(\widehat\xi^{-2}))\log\left(\frac{\widehat
x+\widehat\xi}{\widehat x-\widehat\xi}\right)^2d\widehat\xi=\widehat
x\int_{1/\widehat x}^{1/x} \left(1+\frac{1}{\widehat
x^2t^2}\right)\log\left(\frac{1+t}{1-t}\right)^2dt=4\widehat
x(\log(1/x)+O(1))
\]
and we have
\[
\int_0^{1} y_\lambda'K_2(x,\tau,y_\lambda)d\tau= 4x
(\log(1/x)+O(1)), \quad \lambda\to 0
\]
\smallskip
Summarizing, we get the uniform bound
\begin{equation}\label{dvushka}
\int_0^1 y_\lambda'K_2(x,\tau,y_\lambda)d\tau= \left\{
\begin{array}{cc}
4x(\log(1/\lambda)+O(1)),&\quad x<\lambda\\
4x(\log(1/x)+O(1)), &\quad x>\lambda
\end{array}
\right.
\end{equation}
For the third term in the expression for $A_2$, we have
\[
B_2=B_2^{(1)}+B_2^{(2)}
\]

\[
B_2^{(1)}= \frac{2x}{\sqrt{x^2+\lambda^2}}\int_0^1 \left(x-
\frac{\xi\sqrt{\lambda^2+x^2}}{\sqrt{\lambda^2+\xi^2}}\right) \left(
\frac{y_\lambda(x)+y_\lambda(\xi)}{(x+\xi)^2+(y_\lambda(x)+y_\lambda(\xi))^2}
\right.\]

\[\left.
-\frac{y_\lambda(x)-y_\lambda(\xi)}{(x-\xi)^2+(y_\lambda(x)-y_\lambda(\xi))^2}
\right)d\xi
\]
\[
B_2^{(2)}=\frac{2x}{\sqrt{x^2+\lambda^2}}\int_0^1 \left(x+
\frac{\xi\sqrt{\lambda^2+x^2}}{\sqrt{\lambda^2+\xi^2}}\right) \left(
\frac{y_\lambda(x)+y_\lambda(\xi)}{(x-\xi)^2+(y_\lambda(x)+y_\lambda(\xi))^2}
\right.
\]
\[
\left.
-\frac{y_\lambda(x)-y_\lambda(\xi)}{(x+\xi)^2+(y_\lambda(x)-y_\lambda(\xi))^2}
\right)d\xi
\]
Rescale the variables and recall the formulas (\ref{miracle1}) and
(\ref{miracle2}).

One gets
\begin{equation}\label{fo1}
B_2^{(1)}=-\lambda \frac{4\widehat x}{\sqrt{\widehat x^2+1}}
\int_0^{1/\lambda}\frac{\widehat \xi
r_1}{\sqrt{1+\widehat\xi^2}\Bigl(\widehat
x\sqrt{1+\widehat\xi^2}+\widehat \xi\sqrt{1+\widehat
 x^2}\Bigr)(1+r_1^2)}d\widehat\xi
\end{equation}
As before, we consider two cases.\smallskip

 (1). $\widehat x\in [0,1]$. For the integral over $[0,1]$,
\[
0\leq \int_0^{1}\frac{\widehat \xi
r_1}{\sqrt{1+\widehat\xi^2}\Bigl(\widehat
x\sqrt{1+\widehat\xi^2}+\widehat \xi\sqrt{1+\widehat
x^2}\Bigr)(1+r_1^2)}d\widehat\xi\lesssim 1
\]
The other integral allows the estimate
\[
\int_1^{1/\lambda}\frac{\widehat \xi
r_1}{\sqrt{1+\widehat\xi^2}\Bigl(\widehat
x\sqrt{1+\widehat\xi^2}+\widehat \xi\sqrt{1+\widehat
x^2}\Bigr)(1+r_1^2)}d\widehat\xi\lesssim \log(1/\lambda)
\]
since $r_1\leq 1$.\smallskip

(2). $\widehat x\in [1,1/\lambda]$. We can write
\begin{equation}\label{ang2}
0\leq \int_0^{1/\lambda}\frac{\widehat \xi
r_1}{\sqrt{1+\widehat\xi^2}\Bigl(\widehat
x\sqrt{1+\widehat\xi^2}+\widehat \xi\sqrt{1+\widehat
x^2}\Bigr)(1+r_1^2)}d\widehat\xi\lesssim
\frac{\log(1/\lambda)}{\widehat x}
\end{equation}

\bigskip

For $B_2^{(2)}$, we have similarly

\begin{equation}\label{fo2}
B_2^{(2)}=\frac{4\lambda\widehat x}{\sqrt{\widehat
x^2+1}}\int_0^{1/\lambda}\frac{\widehat
x\sqrt{\widehat\xi^2+1}+\widehat \xi\sqrt{\widehat
x^2+1}}{\sqrt{\widehat \xi^2+1}}\cdot \frac{\widehat\xi
r_1}{(\widehat x+\widehat\xi)^2+r_1^2(\widehat x-\widehat
\xi)^2}d\widehat\xi
\end{equation}
(1). If $\widehat x\in [0,1]$, we get
\[
r_1\lesssim \widehat x+\widehat\xi
\]
and therefore
\[
0<\int_0^{1}\frac{\widehat x\sqrt{\widehat\xi^2+1}+\widehat
\xi\sqrt{\widehat x^2+1}}{\sqrt{\widehat \xi^2+1}}\cdot
\frac{\widehat\xi r_1}{(\widehat x+\widehat\xi)^2+r_1^2(\widehat
x-\widehat \xi)^2}d\widehat\xi\lesssim 1
\]
For the other interval, we use $r_1=1+O(\widehat\xi^{-1})$ to get
\[
\int_1^{1/\lambda}\frac{\widehat x\sqrt{\widehat\xi^2+1}+\widehat
\xi\sqrt{\widehat x^2+1}}{\sqrt{\widehat \xi^2+1}}\cdot
\frac{\widehat\xi r_1}{(\widehat x+\widehat\xi)^2+r_1^2(\widehat
x-\widehat \xi)^2}d\widehat\xi=\frac{\widehat x+\sqrt{\widehat
x^2+1}}{2}\log(1/\lambda)+O(1)
\]
(2). If $\widehat x\in [1,1/\lambda]$, then the asymptotics of $r_1$
yields

\begin{equation}\label{ang1}
\int_0^{1/\lambda}\frac{\widehat x\sqrt{\widehat\xi^2+1}+\widehat
\xi\sqrt{\widehat x^2+1}}{\sqrt{\widehat \xi^2+1}}\cdot
\frac{\widehat\xi r_1}{(\widehat x+\widehat\xi)^2+r_1^2(\widehat
x-\widehat \xi)^2}d\widehat\xi\sim  \widehat x\int_0^{1/\lambda}
\frac{\widehat\xi}{\widehat x^2+\widehat \xi^2}d\widehat\xi\sim
\widehat x\log^+x
\end{equation}

Now, the formulas (\ref{fo1}) and (\ref{fo2}) imply that
$B_2^{(2)}\geq 0$ and $B_2^{(1)}\leq 0$. However,
$B_2=B_2^{(2)}+B_2^{(1)}\geq 0$.  Indeed,  this follows from \eqref{fo1},
\eqref{fo2}, and an estimate
\[
\frac{\widehat x\sqrt{\widehat\xi^2+1}+\widehat\xi\sqrt{1+\widehat
x^2}}{(\widehat x+\widehat \xi)^2+r_1^2(\widehat
x-\widehat\xi)^2}\geq  \frac{1}{(\widehat
x\sqrt{\widehat\xi^2+1}+\widehat\xi\sqrt{1+\widehat x^2)}(1+r_1^2)}
\]
Thus, we have
\[
0\leq B_2\leq B_2^{(2)}\lesssim x\log^+\lambda, \quad 0<x<\lambda
\]
and
\[
0\leq B_2\leq B_2^{(2)}\lesssim x\log^+x, \quad \lambda <x<1
\]
Moreover, \eqref{ang2} and \eqref{ang1} provide a lower bound
\[
B_2\geq \lambda( C_1\widehat x\log^+x-C_2\widehat x^{-1}\log^+\lambda), \quad
x>\lambda
\]
and therefore
\begin{equation}\label{szadi}
B_2\geq C_3 x\log^+x
\end{equation}
for $\widehat x>C_4$ where $C_4$ is sufficiently large absolute
constant.\smallskip

 Consider the sum of the first two terms in \eqref{a2}. We have
\[
\int_0^1 K_1(x,\tau, y_\lambda)d\tau-
\frac{x}{\sqrt{\lambda^2+x^2}}\int_0^1
y'_\lambda(\tau)K_2(x,\tau,y_\lambda ) d\tau
\]
\begin{equation}\label{raz-a}
=4\log(1/\lambda)(\sqrt{x^2+\lambda^2}-x)+O(\lambda), \quad
0<x<\lambda
\end{equation}
and
\begin{equation}\label{raz-b}
=4\log(1/x)(\sqrt{x^2+\lambda^2}-x)+O(x), \quad \lambda<x<\delta
\end{equation}
Add $B_2$ to this expression and divide by
$x\sqrt{x^2+\lambda^2}\log^+(x^2+\lambda^2)$.
On the interval $x\in (0,C_4\lambda)$, we use \eqref{raz-a} and $B_2\geq 0$ to get
$A_2\sim x^{-1}$. For $x\in (C_4\lambda,\delta)$, we apply \eqref{raz-b} and \eqref{szadi} to produce that same bound.
If $x\in[\delta,1]$, we have
convergence to $A_2(x,0)$ which is positive.

\end{proof}
Similar to lemma \ref{auxi}, we have
 \begin{lemma} \label{auxi1} Suppose $\|g(x)-x\|_{\dot Lip[0,1]}\leq \delta\ll
 1$. Then,
\[
\left|\frac{\displaystyle \int_0^1\Bigl(\sqrt{\lambda^2+g^2(\tau})\Bigr)'
K_2(x,\tau,\sqrt{\lambda^2+g^2(\tau)})d\tau}{x\log^+(x^2+\lambda^2)}\right|\lesssim
1
\]
uniformly in $x\in (0,1]$, $\lambda\in (0,1]$, and $g$.
 \end{lemma}
This result can be proved directly or by comparison to the case when $g=x$
if the stability estimates (see \eqref{staaa} below) are used.

\bigskip

{\bf 3. The kernel $M(x,\xi,\lambda)$ and the corresponding
operator}\bigskip

In this subsection, we will show that $M(x,\xi,\lambda)\to M(x,\xi,0)$ in a suitable
sense when $\lambda\to 0$. Recall that $\cal{M}_\lambda$ is the integral operator with the kernel $M(x,\xi,\lambda)$.
We have
\begin{lemma} \label{lemka-6}Fix any $\delta>0$. Then,
\[
\lim_{\lambda\to 0}\sup_{x>\delta} \int_0^1 |M(x,\xi,\lambda)-M(x,\xi,0)|d\xi=0
\]
and therefore
\[
\lim_{\lambda\to 0}\|\omega_\delta^c(x)({\cal M_\lambda-\cal M_0})\|_{L^\infty[0,1],L^\infty[0,1]}=
0
 \]
\end{lemma}
\begin{proof}  We start with
\[
\lim_{\lambda\to 0}\sup_{x>\delta} \int_0^1
|D_2(x,\xi,\lambda)-D_2(x,\xi,0)|d\xi=0
\]
By (\ref{miracle-1}),
\[
\int_0^1 |D_2(x,\xi,\lambda)-D_2(x,\xi,0)|d\xi<C(\delta)\int_0^1
\left(1-\frac{\xi}{\sqrt{\xi^2+\lambda^2}}\right)\log\left|\frac{x+\xi}{x-\xi}\right|d\xi
\]
and the last expression tends to zero uniformly in $x\in [\delta,1]$
when $\lambda\to 0$.\bigskip

 To handle $D_1$, we only need to show that
\begin{equation}
\lim_{\lambda\to 0}\sup_{x\in [\delta,1], \xi\in
[0,1]}\left|\int_\xi^1 D_1(x,\tau,\lambda)d\tau-\int_\xi^1
D_1(x,\tau,0)d\tau\right|=0
\end{equation}
To this end, we first simplify the expression for
$D_1(x,\tau,\lambda)$ using the formulas (\ref{miracle1}) and
(\ref{miracle2}).
\begin{equation}\label{de1}
D_1(x,\xi,\lambda)=D_1^{(1)}+D_1^{(2)}+D_1^{(3)}
\end{equation}
where (below $x=\lambda \widehat x$ and $\xi=\lambda\widehat\xi$)
\[
D_1^{(1)}(x,\xi,\lambda)=\frac{1}{x\sqrt{x^2+\lambda^2}\log^+(x^2+\lambda^2)}\cdot
\frac{4\widehat \xi \widehat x
r_1}{(1+r_1^2)(\widehat\xi\sqrt{1+\widehat x^2}+\widehat
x\sqrt{1+\widehat\xi^2})(1+\widehat\xi^2)}
\]
\[
D_1^{(2)}(x,\xi,\lambda)=\frac{1}{x\sqrt{x^2+\lambda^2}\log^+(x^2+\lambda^2)}\cdot
\left( \frac{\widehat
x\sqrt{1+\widehat\xi^2}+\widehat\xi\sqrt{1+\widehat
x^2}}{1+\widehat\xi^2}\right)\cdot \left( \frac{4\widehat
x\widehat\xi r_1}{(\widehat x+\widehat\xi)^2+(\widehat
x-\widehat\xi)^2r_1^2} \right)
\]
\[
D_1^{(3)}(x,\xi,\lambda)=-\frac{1}{x\log^+(x^2+\lambda^2)}\cdot
\frac{\lambda^2}{(\lambda^2+\xi^2)^{3/2}}\log\left(\frac{x+\xi}{x-\xi}\right)^2
\]

Since $D_1^{(3)}(x,\xi,0)=0$, we first show that
\[
\sup_{x>\delta}\int_0^1 |D_1^{(3)}(x,\xi,\lambda)|d\xi\to 0,
\quad\lambda\to 0
\]
To see that, first split the integral
\[
\int_0^1\frac{\lambda^2}{(\lambda^2+\xi^2)^{3/2}}\log\left(\frac{x+\xi}{x-\xi}\right)^2d\xi
=\int_0^{\delta/2}\frac{\lambda^2}{(\lambda^2+\xi^2)^{3/2}}\log\left(\frac{x+\xi}{x-\xi}\right)^2d\xi
\]
\[
+\int_{\delta/2}^1\frac{\lambda^2}{(\lambda^2+\xi^2)^{3/2}}\log\left(\frac{x+\xi}{x-\xi}\right)^2d\xi
\]
The second integral goes to zero as $\lambda\to 0$  uniformly in
$x>\delta$. The first one is bounded by
\[
C\int_0^{\delta/2}
\frac{\xi\lambda^2}{(\lambda^2+\xi^2)^{3/2}}d\xi\lesssim \lambda
\]
All constants involved are $\delta$ dependent.

Similarly, $D_1^{(1)}(x,\xi,0)=0$ and we have
\[
\sup_{x>\delta}\int_0^1 D_1^{(1)}(x,\xi,\lambda)d\xi\lesssim
\lambda+ \lambda\int_1^\infty \frac{\widehat x\widehat \xi
d\widehat\xi}{(\widehat x\widehat\xi)(1+\widehat\xi^2)}\lesssim
\lambda
\]

For $D_1^{(2)}(x,\xi,0)$, we have
\[
D_1^{(2)}(x,\xi,0)=\frac{1}{x^2\log^+(x^2)}\cdot
\frac{4x^2}{x^2+\xi^2}
\]
To show that
\[
\lim_{\lambda\to
0}\sup_{x>\delta,\xi>0}\int_{\xi}^1|D_2^{(2)}(x,\tau,\lambda)-D_2^{(2)}(x,\tau,0)|d\tau=0
\]
it is sufficient to prove
\[
\lim_{\lambda\to 0} \sup_{x>\delta} \lambda\int_0^{1/\lambda}
\left|\left( \frac{\widehat
x\sqrt{1+\widehat\xi^2}+\widehat\xi\sqrt{1+\widehat
x^2}}{1+\widehat\xi^2}\right)\cdot \left( \frac{4\widehat
x\widehat\xi r_1}{(\widehat x+\widehat\xi)^2+(\widehat
x-\widehat\xi)^2r_1^2} \right)-\frac{4\widehat x^2}{\widehat
x^2+\widehat\xi^2}\right|d\widehat\xi=0
\]
The integral over any interval $[0,T]$ is uniformly bounded. For
large $\widehat x$ and $\widehat \xi$, we substitute
\[
r_1=1+O\left(\frac{1}{\widehat x\widehat\xi}\right),
\sqrt{1+\widehat\xi^2}=\widehat\xi+O(\widehat\xi^{-1}),\sqrt{1+\widehat
x^2}=\widehat x+O(\widehat x^{-1})
\]
Collecting the errors produced by this substitution, we estimate this expression by
\[
\lambda\int_1^{1/\lambda} \frac{\widehat x^2}{\widehat
x^2+\widehat\xi^2}(\widehat\xi^{-2}+\widehat
x^{-2})d\widehat\xi\lesssim \lambda
\]
\end{proof}
The next step is to estimate
\[
\|\omega_\delta(x)\cal M_\lambda \|_{L^\infty[0,1], L^\infty[0,1]}
\]
where $\delta$ and $\lambda$ are small.
\begin{lemma}\label{lemka-7}
We have
\begin{equation}\label{predel-dva}
\lim_{\delta\to 0, \lambda\to 0 }\|\omega_\delta(x)\cal
M_\lambda \|_{L^\infty[0,1], L^\infty[0,1]}=0
\end{equation}
\end{lemma}
\begin{proof}

 We only need to show that
\begin{equation}\label{stupen1}
\lim_{\delta\to 0, \lambda\to 0}\sup_{x\in
[0,\delta]}\int_0^1 \left| D_2(x,\xi,\lambda)+\int_\xi^1
D_1(x,\tau,\lambda)d\tau\right|d\xi=0
\end{equation}
It is instructive to first do that calculation for $\lambda=0$. In this case,
\[
\frac{1}{x\log^+x} \int_0^1 \left|\left( 2\log\left|
\frac{x+\xi}{x-\xi}\right|-\int_\xi^1
\frac{4x}{x^2+\tau^2}d\tau\right)\right|d\xi
\]
\[
=\frac{2}{\log^+x}\int_0^{1/x} \left|\left( \log\left|
\frac{1+\xi}{1-\xi}\right|-\int_\xi^{1/x}
\frac{2}{1+\tau^2}d\tau\right)\right|d\xi
\]
We have
\[
\int_\xi^{1/x} \frac{2}{1+\tau^2}d\tau=\int_\xi^\infty
\frac{2}{1+\tau^2}d\tau+O(x)=\frac{2}{\xi}+O(\xi^{-2}+x), \quad
\xi\gg 1
\]
and
\[
\log\left| \frac{1+\xi}{1-\xi}\right|=\frac{2}{\xi}+O(\xi^{-2})
\]
This entails the necessary cancelation and a bound
\[
\frac{1}{x\log^+x} \int_0^1 \left|\left( 2\log\left|
\frac{x+\xi}{x-\xi}\right|-\int_\xi^1
\frac{4x}{x^2+\tau^2}d\tau\right)\right|d\xi\lesssim
\frac{1}{\log^+x}
\]
The logarithm in the denominator will give convergence to zero when
$x\to 0$.\smallskip

Now, we will need to prove analogous inequalities uniformly in small
$\lambda$. The expression
\[
\int_0^1 \left| D_2(x,\xi,\lambda)+\int_\xi^1
D_1(x,\tau,\lambda)d\tau\right|d\xi
\]
will be handled term by term.

We start by proving
\begin{equation}\label{ohoh1}
\lim_{\delta\to 0, \lambda\to 0}\sup_{x\in
[0,\delta]}\int_0^1 \int_\xi^1 |D_1^{(3)}(x,\tau,\lambda)|d\tau
d\xi=0
\end{equation}
The integral is bounded by
\[
\frac{1}{\widehat x \log^+(\lambda^2\widehat
x^2+\lambda^2)}\int_0^{1/\lambda}\int_{\widehat\xi}^{1/\lambda}\frac{1}{1+\widehat\tau^3}
\log \left|\frac{\widehat x+\widehat\tau}{\widehat
x-\widehat\tau}\right|d\widehat\tau d\widehat\xi
\]
\[
=\frac{1}{\widehat x \log^+(\lambda^2\widehat
x^2+\lambda^2)}\int_0^{1/\lambda}\frac{\widehat\tau}{1+\widehat\tau^3}\log
\left|\frac{\widehat x+\widehat\tau}{\widehat
x-\widehat\tau}\right|d\widehat\tau
\]
For the integral, an estimate holds
\[
\int_0^{1/\lambda}\frac{\widehat\tau}{1+\widehat\tau^3}\log
\left|\frac{\widehat x+\widehat\tau}{\widehat
x-\widehat\tau}\right|d\widehat\tau\lesssim \widehat
x+\int_{1/\widehat x}^{1/x} \widehat x^{-1}u^{-2}\log
\left|\frac{1+u}{1-u}\right|du
\]
The last integral is bounded by $C\widehat x$ for $\widehat x<1$.
For $\widehat x>1$, it is estimated by
$
\displaystyle C\frac{\log^+\widehat x}{\widehat x}
$.
Since
\[
\lim_{\lambda\to 0}\sup_{\widehat x\in (0,1)}\frac{\widehat
x}{\widehat x \log^+(\lambda^2\widehat x^2+\lambda^2)}\lesssim \lim_{\lambda\to 0} \frac{1}{\log^+\lambda}=0
\]
and
\[
\lim_{\lambda\to 0}\sup_{\widehat x>1}\frac{\log^+\widehat
x}{\widehat x^2\log^+(\lambda^2\widehat x^2+\lambda^2)}\lesssim \lim_{\lambda\to 0} \frac{1}{\log^+\lambda}=0
\]
we get (\ref{ohoh1}).\bigskip

Consider the other terms
\[
\int_0^1 \left| D_2(x,\xi,\lambda)+\int_\xi^1
\Bigl(D_1^{(1)}(x,\tau,\lambda)+D_1^{(2)}(x,\tau,\lambda)\Bigr)d\tau\right|d\xi
\]
\[
\lesssim \frac{1}{\widehat x\sqrt{\widehat
x^2+1}\log^+(\lambda^2\widehat x^2+\lambda^2)}\left(
\int_0^{1/\lambda} \left|\frac{\sqrt{\widehat
x^2+1}}{\sqrt{\widehat\xi^2+1}}\widehat\xi \log\left( \frac{\widehat
x+\widehat\xi}{\widehat x-\widehat\xi} \right)^2-\right.\right.
\]
\[
 -\int_{\widehat\xi}^{1/\lambda} \left(\frac{4\widehat\tau\widehat
x r_1}{(1+r_1^2)(1+\widehat\tau^2)(\widehat\tau\sqrt{1+\widehat
x^2}+\widehat x\sqrt{1+\widehat\tau^2})} \right.
\]
\[
\left.\left.\left. +\frac{\widehat
x\sqrt{1+\widehat\tau^2}+\widehat\tau\sqrt{1+\widehat
x^2}}{1+\widehat\tau^2}\cdot \frac{4\widehat x\widehat \tau
r_1}{(\widehat x+\widehat\tau)^2+(\widehat x-\widehat\tau)^2r_1^2}
\right)d\widehat\tau\right|\,d\widehat\xi\right)
\]

We consider two cases.

(1). Take $\widehat x\in (0,1]$. First, let
$\widehat\xi\in (0,1)$. We get
\[
 \int_0^{1} \left|\frac{\sqrt{\widehat
x^2+1}}{\sqrt{\widehat\xi^2+1}}\widehat\xi \log\left( \frac{\widehat
x+\widehat\xi}{\widehat x-\widehat\xi} \right)^2-\right.
\int_{\widehat\xi}^{1/\lambda} \left(\frac{4\widehat\tau\widehat
x r_1}{(1+r_1^2)(1+\widehat\tau^2)(\widehat\tau\sqrt{1+\widehat
x^2}+\widehat x\sqrt{1+\widehat\tau^2})} \right.
\]
\[
\left.\left. +\frac{\widehat
x\sqrt{1+\widehat\tau^2}+\widehat\tau\sqrt{1+\widehat
x^2}}{1+\widehat\tau^2}\cdot \frac{4\widehat x\widehat \tau
r_1}{(\widehat x+\widehat\tau)^2+(\widehat x-\widehat\tau)^2r_1^2}
\right)d\widehat\tau\right|\,d\widehat\xi
\]
\[
\lesssim \widehat x+\widehat x\int_0^1 d\widehat\xi
\int_{\widehat\xi}^{1/\lambda}\left(
\frac{\widehat\tau}{(1+\widehat\tau^2)(\widehat\tau+\widehat
x+\widehat\tau\widehat x)} +
\frac{\widehat\tau(\widehat\tau+\widehat x+\widehat\tau\widehat
x)}{(1+\widehat\tau^2)(\widehat x^2+\widehat
\tau^2)}\right)d\widehat\tau\lesssim \widehat x
\]
Thus, this gives $ O((\log^+\lambda)^{-1}) $ contribution when divided by $\widehat x\sqrt{\widehat x^2+1}\log^+(\lambda^2\widehat x^2+\lambda^2)$. If
$\widehat\xi>1$, we can use the asymptotical formulas
$r_1=1+O(\widehat\xi^{-1})$ and
$\sqrt{\widehat\xi^2+1}=\widehat\xi+O(\widehat\xi^{-1})$ to get
\[
 \int_1^{1/\lambda} \left|\frac{\sqrt{\widehat
x^2+1}}{\sqrt{\widehat\xi^2+1}}\widehat\xi \log\left( \frac{\widehat
x+\widehat\xi}{\widehat x-\widehat\xi} \right)^2-\right.
 \int_{\widehat\xi}^{1/\lambda} \left(\frac{4\widehat\tau\widehat
x r_1}{(1+r_1^2)(1+\widehat\tau^2)(\widehat\tau\sqrt{1+\widehat
x^2}+\widehat x\sqrt{1+\widehat\tau^2})} \right.
\]
\[
\left.\left. +\frac{\widehat
x\sqrt{1+\widehat\tau^2}+\widehat\tau\sqrt{1+\widehat
x^2}}{1+\widehat\tau^2}\cdot \frac{4\widehat x\widehat \tau
r_1}{(\widehat x+\widehat\tau)^2+(\widehat x-\widehat\tau)^2r_1^2}
\right)d\widehat\tau\right|\,d\widehat\xi
\]
\[
=
 \int_1^{1/\lambda} \left|
\frac{4\widehat x\sqrt{\widehat
x^2+1}}{\widehat\xi}(1+O(\widehat\xi^{-2}+\widehat x
\widehat\xi^{-1}))
 -\right.
\]
\[
 \left.-\int_{\widehat\xi}^{1/\lambda} \left(
 \frac{2\widehat x}{\widehat x+\sqrt{\widehat x^2+1}}+2\widehat x(\widehat
 x+\sqrt{1+\widehat
 x^2})\right)\widehat\tau^{-2}+\widehat xO(\widehat\tau^{-3})
 d\widehat\tau\right|d\widehat\xi\lesssim \widehat x
\]
Indeed,
\[
\frac{2\widehat x}{\widehat x+\sqrt{\widehat x^2+1}}+2\widehat
x(\widehat
 x+\sqrt{1+\widehat
 x^2})=4\widehat x\sqrt{1+\widehat x^2}
\]
and we have  cancelation of the main terms.

Summing up these estimates, we get
\begin{equation}\label{stupen2}
\int_0^1 \left| D_2(x,\xi,\lambda)+\int_\xi^1
\Bigl(D_1^{(1)}(x,\tau,\lambda)+D_1^{(2)}(x,\tau,\lambda)\Bigr)d\tau\right|d\xi\lesssim
\frac{1}{\log^+ \lambda}, \quad x\in (0,\lambda)
\end{equation}\bigskip
(2). Consider the case when $\widehat x>1$.
 First, take
$\widehat\xi\in (0,1)$. We get
\[
 \int_0^{1} \left|\frac{\sqrt{\widehat
x^2+1}}{\sqrt{\widehat\xi^2+1}}\widehat\xi \log\left( \frac{\widehat
x+\widehat\xi}{\widehat x-\widehat\xi} \right)^2d\widehat\xi\right|\lesssim 1
\]
and
\[
 \int_0^1\int_{\widehat\xi}^{1/\lambda} \left(\frac{2\widehat\tau\widehat
x r_1}{(1+r_1^2)(1+\widehat\tau^2)(\widehat\tau\sqrt{1+\widehat
x^2}+\widehat x\sqrt{1+\widehat\tau^2})} \right.
\]
\[
\left. +\frac{\widehat
x\sqrt{1+\widehat\tau^2}+\widehat\tau\sqrt{1+\widehat
x^2}}{1+\widehat\tau^2}\cdot \frac{4\widehat x\widehat \tau
r_1}{(\widehat x+\widehat\tau)^2+(\widehat x-\widehat\tau)^2r_1^2}
\right)d\widehat\tau\,d\widehat\xi \lesssim 1+\widehat x
\]
Thus, this gives the contribution bounded by
\[
\sup_{\widehat x>1} \frac{1}{\widehat x\log^+(\lambda^2\widehat
x^2+\lambda^2)}\lesssim \frac{1}{\log^+\lambda}
\]
For the interval $\widehat\xi\in (1,\lambda^{-1})$, we again use
asymptotics for $r_1$,
$\sqrt{\widehat x^2+1}$, and $\sqrt{\widehat\xi^2+1}$:
\[
 \int_1^{1/\lambda} \left|\sqrt{\widehat x^2+1}(1+O(\widehat\xi^{-2}))\log\left( \frac{\widehat
x+\widehat\xi}{\widehat x-\widehat\xi} \right)^2-\right.
\]
\[
 \left.-\int_{\widehat\xi}^{1/\lambda}
\frac{2\widehat x}{\widehat\tau}\left(
\frac{1}{\widehat\tau\sqrt{1+\widehat x^2}+\widehat
x\widehat\tau}+\frac{\widehat\tau\sqrt{1+\widehat x^2}+\widehat
x\widehat\tau}{\widehat x^2+\widehat\tau^2}
\right)(1+O(\widehat\tau^{-1}\widehat x^{-1}+\widehat\tau^{-2}) )
 d\widehat\tau\right|d\widehat\xi
\]
The errors produce the term bounded by
$
C(\log^+\widehat x+\widehat x)
$
and the change of variables in the integrals gives
\[
\int_{1/{\widehat x}}^{1/x} \left|\widehat x\sqrt{1+\widehat
x^2}\log\left(\frac{1+u}{1-u}\right)^2-\int_u^{1/x}2\widehat
x\tau^{-1}\left( \frac{1}{\tau(\sqrt{\widehat x^2+1}+\widehat
x)}+\frac{\tau(\sqrt{\widehat x^2+1}+\widehat
x)}{\tau^2+1}\right)d\tau\right|du
\]
First, notice that
\[
\widehat x\int_{1/{\widehat x}}^{1/x} \left|\int_{1/x}^\infty
2\tau^{-1}\left( \frac{1}{\tau(\sqrt{\widehat x^2+1}+\widehat
x)}+\frac{\tau(\sqrt{\widehat x^2+1}+\widehat
x)}{\tau^2+1}\right)d\tau\right|du\lesssim \widehat x^2
\]
Then,
\[
\int_{1/{\widehat x}}^1 \left|\widehat x\sqrt{1+\widehat
x^2}\log\left(\frac{1+u}{1-u}\right)^2-\widehat
x\int_u^{\infty}2\tau^{-1}\left( \frac{1}{\tau(\sqrt{\widehat
x^2+1}+\widehat x)}+\frac{\tau(\sqrt{\widehat x^2+1}+\widehat
x)}{\tau^2+1}\right)d\tau\right|du\lesssim
\]
\[
\lesssim \widehat x^2+{\log^+\widehat x}
\]
and
\[
\int\limits_{1}^{1/x} \left|\widehat x\sqrt{1+\widehat
x^2}\log\left(\frac{1+u}{1-u}\right)^2-\int\limits_u^{\infty}\frac{2\widehat x}
{\tau}\left( \frac{1}{\tau(\sqrt{\widehat x^2+1}+\widehat
x)}+\frac{\tau(\sqrt{\widehat x^2+1}+\widehat
x)}{\tau^2+1}\right)d\tau\right|du
\lesssim \widehat x^2
\]
after the cancelation of the main terms in the asymptotics.
Collecting these bounds, we get
\[
\int_0^1 \left| D_2(x,\xi,\lambda)+\int_\xi^1
\Bigl(D_1^{(1)}(x,\tau,\lambda)+D_1^{(2)}(x,\tau,\lambda)\Bigr)d\tau\right|d\xi\leq
\frac{\widehat x}{\sqrt{\widehat x^2+1}\log^+(x^2+\lambda^2)}
\]
which (together with \eqref{ohoh1} and \eqref{stupen2}) gives \eqref{stupen1} and finishes the proof.
\end{proof}
We immediately get the following
\begin{corollary}\label{th31}
\[\lim_{\lambda\to
0}\|\cal{M}_\lambda-\cal{M}_0\|_{L^\infty[0,1],L^\infty[0,1]}=0
\]
\end{corollary}
\begin{proof}
It is sufficient to apply lemma \ref{lemka-6} and lemma
\ref{lemka-7}.
\end{proof}

\bigskip

\subsection{ Inverting $L_\lambda$}

Divide the equation
\[
L_\lambda u=g
\]
by $A_1(x,\lambda)$ to rewrite it  as
\[
u'+pu+\int_0^1 M_2(x,\xi,\lambda)u'(\xi)d\xi=g_1
\]
where
\[
p(x,\lambda)=\frac{A_2(x,\lambda)}{A_1(x,\lambda}
\]
and
\[
M_2(x,\xi,\lambda)=\frac{M(x,\xi,\lambda)}{A_1(x,\lambda)}, \quad
g_1=\frac{g(x)}{A_1(x,\lambda)}
\]
Due to (\ref{a10}) and (\ref{a1-1}), this is a minor change as far
as inversion of $L_\lambda$ is concerned.

The equation
\[
u'+pu=F, \quad u(0)=0
\]
has the solution
\[
u=\int_0^x \exp\left(-\int_\xi^x p(t)dt \right) F(\xi)d\xi
\]
and therefore
\[
u'=F-p\int_0^x \exp\left(-\int_\xi^x p(t)dt \right) F(\xi)d\xi
\]
This is the same as
\begin{equation}\label{for2}
u'(x)=g_2(x)-\int_0^1 M_2(x,\xi,\lambda)u'(\xi)d\xi\hspace{4cm}
\end{equation}
\[
\hspace{4cm}+p(x)\int_0^x \exp \left(-\int_t^x
p(\tau)d\tau\right)\int_0^1 M_2(t,\xi,\lambda)u'(\xi)d\xi dt
\]
and
\[
g_2(x)=g_1(x)-p(x)\int_0^x \exp\left(-\int_\xi^x p(t)dt\right)
g_1(\xi)d\xi
\]
We can rewrite
\begin{equation}\label{for3}
u'+O_\lambda u'=B_\lambda g, \quad u'=(I+O_\lambda)^{-1} B_\lambda g
\end{equation}
provided that $I+O_\lambda$ is invertible. The expressions for $O_\lambda$ and $B_\lambda$ are as follows
\[
B_\lambda
g=g_2=\frac{g(x)}{A_1(x,\lambda)}-\frac{A_2(x,\lambda)}{A_1(x,\lambda)}\int_0^x
\exp\left( -\int_\xi^x
\frac{A_2(t,\lambda)}{A_1(t,\lambda)}dt\right)
\frac{g(\xi)}{A_1(\xi,\lambda)}d\xi
\]
and
\begin{equation}\label{ooo}
O_\lambda f=\frac{1}{A_1(x,\lambda)}(\cal{M_\lambda} f)(x)-
\end{equation}
\[
-\frac{A_2(x,\lambda)}{A_1(x,\lambda)}\int_0^x \exp\left(
-\int_\xi^x \frac{A_2(t,\lambda)}{A_1(t,\lambda)}dt\right)
\frac{(\cal{M}_\lambda f)(\xi)}{A_1(\xi,\lambda)}d\xi
\]

\begin{lemma}
We have
\begin{equation}\label{bee}
\|B_\lambda\|_{L^\infty[0,1],L^\infty[0,1]}\lesssim 1
\end{equation}
uniformly in $\lambda\in (0,\lambda_0)$.
\end{lemma}
\begin{proof}
 Since both $A_1$ and $A_2$ are positive, we have
\[
|B_\lambda g(x)|\leq C\left(\|g\|_{L^\infty[0,1]}+\frac{1}{x}\int_0^x
|g(\xi)|d\xi\right)
\]
uniformly in $\lambda\in (0,\lambda_0)$ and $x\in (0,1]$ as follows
from the analysis of $A_1$ and $A_2$. This gives \eqref{bee}.
\end{proof}

Consider $O_\lambda$. We have
\begin{lemma}\label{staba}
\[
\|O_\lambda-O_0\|_{L^\infty[0,1],L^\infty[0,1]}\to 0, \quad \lambda\to 0
\]
\end{lemma}
\begin{proof}

 For the first term,
\[
\left\|\frac{1}{A_1(x,\lambda)}\cal{M_\lambda}f-\frac{1}{A_1(x,0)}\cal{M}_0f\right\|_{L^\infty[0,1]}=o(1)\|f\|_{L^\infty[0,1]},
\]
and $o(1)\to 0$ when $\lambda\to 0$, uniformly in $f$. Indeed, this
 follows from the corollary \ref{th31} and the properties of
$A_1(x,\lambda)$.

The second term can be written as
\[
\omega_\delta(x) \cdot \frac{A_2(x,\lambda)}{A_1(x,\lambda)}\int_0^x
\exp\left( -\int_\xi^x
\frac{A_2(t,\lambda)}{A_1(t,\lambda)}dt\right)
\frac{(\cal{M}_\lambda f)(\xi)}{A_1(\xi,\lambda)}d\xi
\]
\[
+\omega^{c}_\delta(x)\cdot
\frac{A_2(x,\lambda)}{A_1(x,\lambda)}\int_0^x \exp\left( -\int_\xi^x
\frac{A_2(t,\lambda)}{A_1(t,\lambda)}dt\right)
\frac{(\cal{M}_\lambda f)(\xi)}{A_1(\xi,\lambda)}d\xi
\]
where $\delta>0$. If we denote the first/second expressions by
$S_{1(2)}$, then
\[
|S_1|\lesssim \frac{1}{x}\int_0^x
\chi_{\xi<\delta}\left|\frac{(\cal{M}_\lambda
f)(\xi)}{A_1(\xi,\lambda)}\right|d\xi
\]
and
\[
\lim_{\delta\to 0, \lambda\to 0}\|S_1\|_{L^\infty[0,1],L^\infty[0,1]}=0
\]
The last equality follows from \eqref{predel-dva}.

For $S_2$, one can write similarly
\[
S_2=\omega^{c}_\delta(x)\cdot
\frac{A_2(x,\lambda)}{A_1(x,\lambda)}\int_0^x \chi_{\xi<\delta}\cdot
\exp\left( -\int_\xi^x
\frac{A_2(t,\lambda)}{A_1(t,\lambda)}dt\right)
\frac{(\cal{M}_\lambda f)(\xi)}{A_1(\xi,\lambda)}d\xi
\]
\[
+\omega^{c}_\delta(x)\cdot
\frac{A_2(x,\lambda)}{A_1(x,\lambda)}\int_0^x\chi_{\xi>\delta}\cdot
\exp\left( -\int_\xi^x
\frac{A_2(t,\lambda)}{A_1(t,\lambda)}dt\right)
\frac{(\cal{M}_\lambda f)(\xi)}{A_1(\xi,\lambda)}d\xi
\]
The first expression can be handled in the same way. For the second,
we consider
\[
\left\|\omega_\delta^{c}(x)\cdot
\frac{A_2(x,\lambda)}{A_1(x,\lambda)}\int_0^x\chi_{\xi>\delta}\cdot
\exp\left( -\int_\xi^x
\frac{A_2(t,\lambda)}{A_1(t,\lambda)}dt\right)
\frac{(\cal{M}_\lambda f)(\xi)}{A_1(\xi,\lambda)}d\xi\right.
\]
\[
-\left.\omega_\delta^{c}(x)\cdot
\frac{A_2(x,0)}{A_1(x,0)}\int_0^x\chi_{\xi>\delta}\cdot \exp\left(
-\int_\xi^x \frac{A_2(t,0)}{A_1(t,0)}dt\right) \frac{(\cal{M}_0
f)(\xi)}{A_1(\xi,0)}d\xi\right\|_{L^\infty[0,1]}
\]
If $\delta>0$ is fixed, this expression is bounded by
$o(1)\|f\|_{L^\infty[0,1]}$ as $\lambda\to 0$ (with constant depending on
$\delta$). That follows directly from the properties of $A_{1(2)}$
and $\cal{M}_\lambda$.  Combining the obtained estimates we get the statement of the lemma.\end{proof}

By the standard argument of the perturbation theory, this lemma implies that inversion of  $I+O_\lambda$ can be reduced
to showing that $I+O_0$ is invertible. In the next section, we will
check that.\bigskip

\subsubsection{The operator $O_0$ and its properties}\bigskip

\begin{theorem}\label{oyoy}
The operator $I+O_0$ is invertible in $L^\infty[0,1]$.
\end{theorem}
\begin{proof}
For the case $\lambda=0$, the formulas are very simple. We recall
that
\[
K_1(x,\xi,y_0)=K_2(x,\xi,y_0)=\log\left(\frac{x+\xi}{x-\xi}\right)^2
\]
Then,  \eqref{a10} and \eqref{a2-1} imply that
\[
A_1(x,0)=\frac{1}{\log^+x^2}\int_0^{1/x} \log\left(
\frac{1+\xi}{1-\xi} \right)^2d\xi=2+o(1), \quad x\to 0
\]
and
\[
A_2(x,0)=\frac{2}{x\log^+x^2}\log(x^{-2}+1)=\frac{2}{x}+o(x), \quad
x\to 0
\]
Thus,
\[
p(x,0)=\frac{A_2(x,0)}{A_1(x,0)}= \frac{1}{x}+o(x), \quad x\to 0
\]
Then,
\[
D_2(x,\xi)=-\frac{1}{x\log^+x^2}\log\left(\frac{x+\xi}{x-\xi}\right)^2
\]
and
\[
D_1(x,\xi)=\frac{4}{(x^2+\xi^2)\log^+ x^2}
\]
Therefore,
\[
M(x,\xi,0)=\frac{1}{x\log^+x^2}\left( -\log\left(
\frac{x+\xi}{x-\xi}
\right)^2+\int_\xi^1\frac{4x}{x^2+\tau^2}d\tau\right)
\]
and
\[
M_2(x,\xi,0)=\frac{M(x,\xi,0)}{A_1(x,0)}
\]
where $A_1\sim 1$ on all of $[0,1]$.

\begin{lemma}\label{lemka-c}
The operator
\[
G_2f=\int_0^1 M_2(x,\xi,0)f(\xi)d\xi
\]
is compact in $L^\infty[0,1]$.
\end{lemma}
\begin{proof} First, notice that
\begin{equation}\label{vova1}
|G_2f|\lesssim \frac{\|f\|_{\infty}}{\log^+ x}\int_0^{1/x} \left|
\log\left(     \frac{1+\xi}{1-\xi}
\right)^2-4\int_{\xi}^{1/x}\frac{d\tau}{\tau^2+1}
\right|d\xi\lesssim \frac{\|f\|_\infty}{\log^+ x}
\end{equation}
and thus $G_2$ is bounded in $L^\infty[0,1]$.

 The compactness now follows by the standard approximation
argument. Let us write a partition of unity
$1=\phi_\delta+\phi^c_\delta$. Then, \eqref{vova1} yields $\|\phi_\delta G_2\|_{L^\infty[0,1],L^\infty[0,1]}\to 0$ as
$\delta\to 0$. Then, for fixed $\delta>0$, $\phi^c_\delta G_2$ is compact since the kernel
has a weak singularity on the diagonal and is smooth away from it. Since the space of compact operators is closed in the operator topology, we have the statement of the lemma.
\end{proof}
For $O_0$, one gets
\begin{equation}\label{ooo1}
O_0 f=\frac{1}{A_1(x,0)}\cal{M}_0 f
-\frac{A_2(x,0)}{A_1(x,0)}\int_0^x \exp\left( -\int_\xi^x
\frac{A_2(t,0)}{A_1(t,0)}dt\right)\frac{(\cal{M}_0
f)(\xi)}{A_1(\xi,0)} d\xi
\end{equation}
\[
=(G_2 f)(x)-
\frac{A_2(x,0)}{A_1(x,0)}\int_0^x \exp\left( -\int_\xi^x
\frac{A_2(t,0)}{A_1(t,0)}dt\right)(G_2f)(\xi) d\xi
\]
Since the operator $G_3$ defined by
\[
G_3f=\frac{A_2(x,0)}{A_1(x,0)}\int_0^x \exp\left( -\int_\xi^x
\frac{A_2(t,0)}{A_1(t,0)}dt\right)f(\xi)d\xi
\]
is bounded in $L^\infty[0,1]$, we get the compactness for $O_0$ in view of lemma \ref{lemka-c}. Therefore,
the Fredholm theory is applicable to $I+O_0$. In particular, to
prove invertibility of $I+O_0$, we only need to check that its
kernel is trivial.

Consider the equation
\[
(I+O_0)f=0
\]
and suppose that $f\in L^\infty[0,1]$. Recall \eqref{for3}. The equation
\begin{equation}\label{prosti}
L_0u=0,\quad u\in \dot Lip [0,1]
\end{equation}
is equivalent to
\[
(I+O_0)u'=0, \quad u(x)=\int_0^x f(t)dt
\]
Thus, we only need to check that $L_0$
has zero kernel in $\dot Lip[0,1]$.

The equation \eqref{prosti} is equivalent to
\[
\int_0^1 (u'(x)-u'(\xi))K_1(x,\xi,y_0)d\xi+8\int_0^1 H'(2x^2+2\xi^2)
(\xi u(x)+xu(\xi))d\xi=0, \quad u\in \dot Lip[0,1]
\]
where
\[
K_1(x,\xi,y_0)=H(2(x+\xi)^2)-H(2(x-\xi)^2)= \log\left(
\frac{x+\xi}{x-\xi}\right)^2
\]
since
$
H(x)=\log x
$
in that case. Multiply the
both sides by $u$ and integrate over $[0,1]$. For the general $H$,
we have
\[
0.5 \int_0^1 (u(1)-u(\xi))^2 \Bigl(H(2(1+\xi)^2)-H(2(1-\xi)^2)
\Bigr)d\xi
\]
\[
-2\int_0^1\int_0^1 (u(x)-u(\xi))^2 \Bigl(
H'(2(x+\xi)^2)(x+\xi)-H'(2(x-\xi)^2)(x-\xi)\Bigr)\Bigr)dxd\xi
\]
\[
+8\int_0^1u^2(x)\int_0^1 \xi H'(2x^2+2\xi^2)d\xi
dx+8\int_0^1\int_0^1 u(x)u(\xi)xH'(2\xi^2+2x^2)dxd\xi
\]
\[
=\mathfrak{I}_1+\ldots+\mathfrak{I}_4
\]
\smallskip
Let us study this expression term by term.

If $u_1(x)=u(1)-u(x)$, then
\[
\mathfrak{I}_1=\int_0^1 u_1^2(x)\log \left| \frac{1+x}{1-x}\right|dx\geq 0
\]
This is actually true for generic $H$ that are monotonically
increasing.\bigskip

Using the symmetrization of the
integrals, we get the following expressions
\[
\mathfrak{I}_2=-\int_0^1\int_0^1
\frac{(u(x)-u(\xi))^2}{x+\xi}dxd\xi=-2\int_0^1u^2(x)\log
\left(\frac{1+x}{x}\right)dx
\]
\[
+2\int_0^1\int_0^1 \frac{u(x)u(\xi)}{x+\xi}dxd\xi
\]
\[
\mathfrak{I}_3=2\int_0^1 u^2(x)\log(1+x^{-2})dx
\]
\[
\mathfrak{I}_4=4\int_0^1\int_0^1
u(x)u(\xi)\frac{x}{x^2+\xi^2}dxd\xi=2\int_0^1\int_0^1
u(x)u(\xi)\frac{x+\xi}{x^2+\xi^2}dxd\xi
\]
Notice now that the sum of the first term in $\mathfrak{I}_2$ and $\mathfrak{I}_3$ is
\[
2\int_0^1 u^2(x)\log\left(\frac{x+x^{-1}}{1+x}\right)dx\geq 0
\]
because
$
x+x^{-1}\geq x+1
$
if $x\in (0,1]$.

In the calculations that follow, the condition $u(x)=O(x), \, x\to
0$ will ensure the convergence of all integrals involved. Since the Hilbert
matrix is nonnegative (\cite{nik}, proof of the theorem 5.3.1.), the integral
\[
{\cal G}(u)=\int_0^1 \frac{u(\xi)}{x+\xi}d\xi
\]
defines a  positive definite operator in $L^2(0,1)$. Thus,
\[
{\cal G}_1(u)=\int_0^1 \frac{u(\xi)}{x^2+\xi^2}d\xi
\]
is positive definite as well, as the change of variables in the
quadratic form shows. Also,
\[
\frac{x+\xi}{x^2+\xi^2}=\frac{1}{x+\xi}+\frac{2x\xi}{(x^2+\xi^2)(x+\xi)}
\]
So, we only need to establish that
\[
{\cal G}_2u=\int_0^1 \frac{x\xi u(\xi)}{(x^2+\xi^2)(x+\xi)}d\xi
\]
is positive definite. That, however, is the corollary of the Schur's
theorem for the Hadamard product of the positive definite matrices (\cite{nik}, p.319), written
for the integral operators (e.g., by the Riemann sum approximation). Indeed,
it is sufficient to notice that
\[
\frac{x\xi}{x^2+\xi^2}
\]
is a positive definite kernel (again, by the change of variables in
the quadratic form).
\end{proof}
Summing up the results of this section, we obtain
(\ref{aaa1}).\bigskip

\section{$\|\psi_0\|_{\dot Lip[0,1]}$ is small.}

In this section, we will prove \eqref{aaa4}, the smallness of
initial data for the contraction mapping.
\begin{lemma}
We have
\begin{equation}\label{why-small}
\|\psi_0\|_{\dot Lip[0,1]}=o(1), \quad \lambda\to 0
\end{equation}
\end{lemma}
\begin{proof}
As it follows from the previous section, we only need to show
\begin{equation}\label{selfi}
\|F(x,\lambda)\|_{L^\infty[0,1]}=o(1), \quad \lambda\to 0
\end{equation}
Recall the definition of $F$,
\[
F(x,\lambda)=\frac{1}{x\sqrt{\lambda^2+x^2}\log^+(x^2+\lambda^2)}\left(x\int_0^1
K_1(x,\tau,y_\lambda)d\tau-\sqrt{\lambda^2+x^2}\int_0^1
y_\lambda'(\tau)K_2(x,\tau,y_\lambda)d\tau\right)
\]
For any given $\delta>0$, we clearly have
\[
\lim_{\lambda\to 0}\|\omega_\delta^{c}\cdot F(x,\lambda)\|_{L^\infty[0,1]}=0
\]
For $x<\delta$, we can use the asymptotics established above
(e.g., \eqref{odnushka},  \eqref{treshka}, and \eqref{dvushka}).
This gives
\[
\|\omega_{\delta}\cdot F(x,\lambda)\|_{L^\infty[0,1]}\lesssim
\frac{1}{\log^+(\delta^2+\lambda^2)}
\]
These two estimates finish the proof of the lemma.
\end{proof}

\bigskip

\section{The Frechet differentiability.}

In this section, we study $Q(u)$ given by
\[
Q(u)=F(f,\lambda)-F(x,\lambda)-D_fF(x,\lambda)u, \quad f=x+u
\]
and prove \eqref{aaa2} and \eqref{aaa3}. We assume in this section
that $\lambda\in (0,1)$. Notice first that $Q(0)=0$ and therefore
\eqref{aaa2} follows from \eqref{aaa3}. Let us prove
\eqref{aaa3}.

We write
\[
Q(u_2)=F(x+u_2,\lambda)-F(x,\lambda)-D_fF(x,\lambda)(x+u_2)
\]
\[
Q(u_1)=F(x+u_1,\lambda)-F(x,\lambda)-D_fF(x,\lambda)(x+u_1)
\]
Subtract and write
\begin{eqnarray*}
|Q(u_2)-Q(u_1)|\leq
|F(x+u_2,\lambda)-F(x+u_1,\lambda)-D_fF(x+u_1,\lambda)(u_2-u_1)|
\\
+|\Bigl(D_fF(x+u_1,\lambda)-D_fF(x,\lambda)\Bigr)(u_2-u_1)|
\end{eqnarray*}
for every point $x\in (0,1]$.
Thus, we only have to prove two bounds:
\begin{equation}\label{adyn1}
\|F(x+u_2,\lambda)-F(x+u_1,\lambda)-D_fF(x+u_1,\lambda)(u_2-u_1)\|_{L^\infty[0,1]}=
o(1)\|u_2-u_1\|_{\dot Lip[0,1]}
\end{equation}
and
\begin{equation}\label{dwa}
\|D_fF(x+u,\lambda)-D_fF(x,\lambda)\|_{\dot Lip[0,1],L^\infty[0,1]}=o(1),
\quad \|u\|_{\dot Lip[0,1]}\leq \delta, \quad \delta\to 0
\end{equation}\bigskip

\subsection{The proof of \eqref{adyn1}} We start with proving
\eqref{adyn1}.\bigskip

Denote $\rho(x)=x+u_1(x)$. By our assumptions we have
\[
\|\rho'(x)-1\|_{L^\infty[0,1]}\leq \delta\ll 1, \quad \rho(0)=0
\]
Therefore,
\[
\rho(x)=x\left(1+\int_0^1 (\rho'(xt)-1)dt\right)=x(1+O(\delta))
\]
\bigskip

{\bf Remark.} We will use the following property many times in the
arguments below. Given arbitrary $M>0$, the scaled function
$\rho_M(\widehat x)=M\rho(M^{-1}\widehat x)$ satisfies:
\[
\rho_M(0)=0, \quad \|\rho'_M(\widehat x)-1\|_{L^\infty[0,M]}\leq
\delta
\]
Moreover, if $\|h-g\|_{\dot Lip[0,1]}\leq \epsilon$, then
$\|h_M-g_M\|_{\dot Lip[0,M]}\leq \epsilon$ after scaling.

\bigskip

Take $t\in \mathbb{R}$ with $|t|<t_0=\|u_2-u_1\|_{\dot Lip[0,1]}$ and
$f: \|f\|_{\dot Lip[0,1]}\leq 1$. Consider $f_t(x)=\rho(x)+tf(x)$.
We only need to show that
\begin{equation}\label{adyn}
\|F(f_t,\lambda)-F(f_0,\lambda)-tD_fF(f_0,\lambda)f\|_{L^\infty[0,1]}=to(1),
\quad t\to 0
\end{equation}
uniformly in $f$ and $\rho$.

Fix arbitrary $x\in (0,1]$ and apply the mean-value formula to
$F(f_t,\lambda)-F(f_0,\lambda)$,
\[
F(f_t,\lambda)-F(f_0,\lambda)=t\frac{P_1+\ldots+P_6}{x\sqrt{x^2+\lambda^2}\log^+(x^2+\lambda^2)}
\]
where $t_1(x)\in [0,t]$. Introducing
$Y_{\lambda,t}(x)=\sqrt{\lambda^2+f_t^2(x)}$, we get
\[
P_1=f(\rho'+t_1f')\int_0^1 K_1(x,\tau,Y_{\lambda,t_1})d\tau
\]
\[
P_2=(\rho+t_1f)f'\int_0^1 K_1(x,\tau,Y_{\lambda,t_1})d\tau
\]
\[
P_3=2(\rho+t_1f)(\rho'+t_1f')\left(X_1+\ldots+X_4\right)
\]
where
\[
X_1=\int_0^1 \frac{Y_{\lambda,t_1}(x)+Y_{\lambda,t_1}(\tau)}
{(x+\tau)^2+(Y_{\lambda,t_1}(x)+Y_{\lambda,t_1}(\tau))^2} \left(
\frac{f_{t_1}(x)f(x)}{Y_{\lambda,t_1}(x)}+\frac{f_{t_1}(\tau)f(\tau)}{Y_{\lambda,t_1}(\tau)}
\right)d\tau
\]
\[
X_2=\int_0^1 \frac{Y_{\lambda,t_1}(x)+Y_{\lambda,t_1}(\tau)}
{(x-\tau)^2+(Y_{\lambda,t_1}(x)+Y_{\lambda,t_1}(\tau))^2} \left(
\frac{f_{t_1}(x)f(x)}{Y_{\lambda,t_1}(x)}+\frac{f_{t_1}(\tau)f(\tau)}{Y_{\lambda,t_1}(\tau)}
\right)d\tau
\]
\[
X_3=-\int_0^1 \frac{Y_{\lambda,t_1}(x)-Y_{\lambda,t_1}(\tau)}
{(x-\tau)^2+(Y_{\lambda,t_1}(x)-Y_{\lambda,t_1}(\tau))^2} \left(
\frac{f_{t_1}(x)f(x)}{Y_{\lambda,t_1}(x)}-\frac{f_{t_1}(\tau)f(\tau)}{Y_{\lambda,t_1}(\tau)}
\right)d\tau
\]
\[
X_4=-\int_0^1 \frac{Y_{\lambda,t_1}(x)-Y_{\lambda,t_1}(\tau)}
{(x+\tau)^2+(Y_{\lambda,t_1}(x)-Y_{\lambda,t_1}(\tau))^2} \left(
\frac{f_{t_1}(x)f(x)}{Y_{\lambda,t_1}(x)}-\frac{f_{t_1}(\tau)f(\tau)}{Y_{\lambda,t_1}(\tau)}
\right)d\tau
\]
and
\[
P_4=-\frac{f_{t_1}f}{Y_{\lambda,t_1}}\int_0^1
Y'_{\lambda,t_1}(\tau)K_2(x,\tau,Y_{\lambda,t_1})d\tau
\]
\[
P_5=-Y_{\lambda,t_1}\int_0^1
\left(\frac{f_{t_1}f}{Y_{\lambda,t_1}}\right)'K_2(x,\tau,Y_{\lambda,t_1})d\tau
\]
\[
P_6=-2Y_{\lambda,t_1}(L_1+\ldots+L_4)
\]
Similarly, for $\{L_j\}$ we have
\[
L_1=\int_0^1
Y_{\lambda,t_1}'(\tau)\frac{Y_{\lambda,t_1}(x)+Y_{\lambda,t_1}(\tau)}
{(x+\tau)^2+(Y_{\lambda,t_1}(x)+Y_{\lambda,t_1}(\tau))^2} \left(
\frac{f_{t_1}(x)f(x)}{Y_{\lambda,t_1}(x)}+\frac{f_{t_1}(\tau)f(\tau)}{Y_{\lambda,t_1}(\tau)}
\right)d\tau
\]
\[
L_2=-\int_0^1
Y_{\lambda,t_1}'(\tau)\frac{Y_{\lambda,t_1}(x)+Y_{\lambda,t_1}(\tau)}
{(x-\tau)^2+(Y_{\lambda,t_1}(x)+Y_{\lambda,t_1}(\tau))^2} \left(
\frac{f_{t_1}(x)f(x)}{Y_{\lambda,t_1}(x)}+\frac{f_{t_1}(\tau)f(\tau)}{Y_{\lambda,t_1}(\tau)}
\right)d\tau
\]
\[
L_3=-\int_0^1
Y_{\lambda,t_1}'(\tau)\frac{Y_{\lambda,t_1}(x)-Y_{\lambda,t_1}(\tau)}
{(x-\tau)^2+(Y_{\lambda,t_1}(x)-Y_{\lambda,t_1}(\tau))^2} \left(
\frac{f_{t_1}(x)f(x)}{Y_{\lambda,t_1}(x)}-\frac{f_{t_1}(\tau)f(\tau)}{Y_{\lambda,t_1}(\tau)}
\right)d\tau
\]
\[
L_4=\int_0^1
Y_{\lambda,t_1}'(\tau)\frac{Y_{\lambda,t_1}(x)-Y_{\lambda,t_1}(\tau)}
{(x+\tau)^2+(Y_{\lambda,t_1}(x)-Y_{\lambda,t_1}(\tau))^2} \left(
\frac{f_{t_1}(x)f(x)}{Y_{\lambda,t_1}(x)}-\frac{f_{t_1}(\tau)f(\tau)}{Y_{\lambda,t_1}(\tau)}
\right)d\tau
\]

 We need to show that
\[
\left\|\frac{(P_1+\ldots+P_6)-(P_1^0+\ldots+P_6^0)}
{x\sqrt{\lambda^2+x^2}\log^+(x^2+\lambda^2)}\right\|_{L^\infty[0,1]}=o(1)
\]
as $t\to 0$ uniformly in $f$ and $\rho$. Here $P_j^0$ are the
similar expressions taken with $t_1=0$.  \vspace{1cm}

{\bf (1).} We start with $P_1-P_1^0$.
\[
\frac{1}{x\sqrt{x^2+\lambda^2}\log^+(x^2+\lambda^2)}\left|f(\rho'+t_1f')\int_0^1
K_1(x,\tau,Y_{\lambda,t_1})d\tau-f\rho'\int_0^1
K_1(x,\tau,Y_{\lambda,0})d\tau\right|
\]
\begin{equation}\label{snoska1}
\leq \frac{t}{\sqrt{x^2+\lambda^2}\log^+(x^2+\lambda^2)}\int_0^1
K_1(x,\tau,Y_{\lambda,t_1})d\tau
\end{equation}
\[
+\frac{1}{\sqrt{x^2+\lambda^2}\log^+(x^2+\lambda^2)} \left|\int_0^1
K_1(x,\tau,Y_{\lambda,t_1})d\tau-\int_0^1
K_1(x,\tau,Y_{\lambda,0})d\tau\right|
\]
To handle the first term, we use the lemma \ref{auxi}. The lemma
\ref{bbb1} below  takes care of the second term.

\begin{lemma}\label{bbb1}
We have
\[
\left\|\frac{\int_0^1 K_1(x,\tau,Y_{\lambda,t_1})d\tau-\int_0^1
K_1(x,\tau,Y_{\lambda,0})d\tau}{\sqrt{x^2+\lambda^2}\log^+(x^2+\lambda^2)}\right\|_{L^\infty[0,1]}=o(1),
\quad t\to 0
\]
uniformly in $\lambda$, $f$, and $\rho$.
\end{lemma}
\begin{proof}
By the mean-value formula we have
\[
\int_0^1 K_1(x,\tau,Y_{\lambda,t_1})d\tau-\int_0^1
K_1(x,\tau,Y_{\lambda,0})d\tau=t_1(x)(\widehat X_1+\ldots+\widehat
X_4)
\]
where the expressions $\widehat X_j$ are different from $X_j$
defined above only by $t_1$ replaced with $t_2$. The bound
\[
\left\|\frac{\widehat X_1+\ldots+\widehat
X_4}{\sqrt{x^2+\lambda^2}\log^+(x^2+\lambda^2)}\right\|_{L^\infty[0,1]}\lesssim
1
\]
follows from the theorem \ref{real-lemma} below.
\end{proof}

{\bf (2).} The term $P_2-P_2^0$ can be handled in exactly the same
way.\bigskip

{\bf (3).} The term $P_3-P_3^0$ is more complicated.\bigskip

Arguing similarly to $P_1$, we only need to prove the following
theorem.
\begin{theorem}\label{real-lemma}
\[
\left\|\frac{ (X_1+\ldots+ X_4)-(X_1^0+\ldots+
X_4^0)}{\sqrt{x^2+\lambda^2}\log^+(x^2+\lambda^2)}\right\|_{L^\infty[0,1]}=o(1),
\quad t\to 0
\]
and
\[
\left\|\frac{ X_1^0+\ldots+
X_4^0}{\sqrt{x^2+\lambda^2}\log^+(x^2+\lambda^2)}\right\|_{L^\infty[0,1]}
\lesssim 1
\]
uniformly in $\lambda$, $f$, and $\rho$.
\end{theorem}
\begin{proof}
Let us introduce $x=\lambda \widehat x$ and
$\tau=\lambda\widehat\tau$. Notice that
\[
Y_{\lambda,t}(\lambda \widehat
x)=\lambda\sqrt{1+(\lambda^{-1}f_t(\lambda \widehat x))^2}
\]
Let us focus of $X_1+X_3$ first. We are going to prove the following
general result. Once we do that, it suffices to apply it to the
scaled $X_1+X_3$  by taking $y_1(x)=f(x)$ and $y_2(x)=f_{t_1}(x)$.
\begin{lemma}\label{hard}
Suppose $y_{1}, y_2, \widetilde y_2\in \dot Lip[0,\lambda^{-1}]$ and
\[
\|y'_1\|_{L^\infty[0,1/\lambda]}\leq 1,\quad \|y'_2-\widetilde y'_2\|_{L^\infty[0,1/\lambda]}\leq
\epsilon, \quad \|y'_2-1\|_{L^\infty[0,1/\lambda]}\ll 1
\]
 If one defines
\[
H=\frac{1}{\sqrt{\widehat x^2+1}\log^+(\lambda^2(\widehat
x^2+1))}\int_0^{1/\lambda} \left(\left(\frac{y_1(\widehat
 x)y_2(\widehat x)}{\sqrt{1+y_2^2(\widehat x)}}+\frac{y_1(\widehat\tau)y_2(\widehat
\tau)}{\sqrt{1+y_2^2(\widehat\tau)}}\right)\right.
\]
\[
\times \frac{\sqrt{1+y_2^2(\widehat
x)}+\sqrt{1+y_2^2(\widehat\tau)}}{(\widehat x+\widehat
\tau)^2+(\sqrt{1+y_2^2(\widehat
x)}+\sqrt{1+y_2^2(\widehat\tau)})^2}
\]
\[
\left. -\left(\frac{y_1(\widehat x)y_2(\widehat
x)}{\sqrt{1+y_2^2(\widehat x)}}-\frac{y_1(\widehat\tau)y_2(\widehat
\tau)}{\sqrt{1+y_2^2(\widehat\tau)}}\right)\times \frac{\sqrt{1+y_2^2(\widehat
x)}-\sqrt{1+y_2^2(\widehat\tau)}}{(\widehat x-\widehat
\tau)^2+(\sqrt{1+y_2^2(\widehat x)}-\sqrt{1+y_2^2(\widehat\tau)})^2}
\right)d\widehat \tau
\]
and
\[
\widetilde H=\frac{1}{\sqrt{\widehat x^2+1}\log^+(\lambda^2(\widehat
x^2+1))}\int_0^{1/\lambda} \left(\left(\frac{y_1(\widehat
 x)\widetilde y_2(\widehat x)}{\sqrt{1+\widetilde y_2^2(\widehat x)}}+\frac{y_1(\widehat\tau)\widetilde y_2(\widehat
\tau)}{\sqrt{1+\widetilde y_2^2(\widehat\tau)}}\right)\right.
\]
\[
\times \frac{\sqrt{1+\widetilde y_2^2(\widehat
x)}+\sqrt{1+\widetilde y_2^2(\widehat\tau)}}{(\widehat x+\widehat
\tau)^2+(\sqrt{1+\widetilde y_2^2(\widehat x)}+\sqrt{1+\widetilde
y_2^2(\widehat\tau)})^2}
\]
\[
\left. -\left(\frac{y_1(\widehat x)\widetilde y_2(\widehat
x)}{\sqrt{1+\widetilde y_2^2(\widehat
x)}}-\frac{y_1(\widehat\tau)\widetilde y_2(\widehat
\tau)}{\sqrt{1+\widetilde
y_2^2(\widehat\tau)}}\right)\times \frac{\sqrt{1+\widetilde y_2^2(\widehat
x)}-\sqrt{1+\widetilde y_2^2(\widehat\tau)}}{(\widehat x-\widehat
\tau)^2+(\sqrt{1+\widetilde y_2^2(\widehat x)}-\sqrt{1+\widetilde
y_2^2(\widehat\tau)})^2} \right)d\widehat \tau
\]
then
\[
\|H-\widetilde H\|_{L^\infty[0,1/\lambda]}=o(1), \quad \epsilon\to 0
\]
and
\[
\|H\|_{L^\infty[0,1/\lambda]}\lesssim 1
\]
uniformly in $\lambda\in (0,1)$, $y_1$, $y_2$, and $\widetilde y_2$.
\end{lemma}
\begin{proof}
We will study $H$ in detail and, in particular, its stability in
$y_2$. That will give the necessary bounds. Notice first that
\begin{equation}\label{stab1}
\left|\frac{y_2(\widehat x)}{\sqrt{y_2^2(\widehat
x)+1}}-\frac{\widetilde y_2(\widehat x)}{\sqrt{\widetilde
y_2^2(\widehat x)+1}}\right|\lesssim \left\{
\begin{array}{cc}
\epsilon \widehat x, &\widehat x<1\\
\epsilon\widehat x^{-2}, &\widehat x>1
\end{array}
\right.
\end{equation}
The second term in the formula for  $H$ has the singularity of
the type $(\widehat x-\widehat\tau)^2$ in the denominator. However,
this is compensated by the zero in the numerator and
\[
\sup_{\widehat x}\left|\int_{|\widehat \tau-\widehat
x|<1}\left(\frac{y_1(\widehat x)y_2(\widehat
x)}{\sqrt{1+y_2^2(\widehat x)}}-\frac{y_1(\widehat\tau)y_2(\widehat
\tau)}{\sqrt{1+y_2^2(\widehat\tau)}}\right)\frac{\sqrt{1+y_2^2(\widehat
x)}-\sqrt{1+y_2^2(\widehat\tau)}}{(\widehat x-\widehat
\tau)^2+(\sqrt{1+y_2^2(\widehat x)}-\sqrt{1+y_2^2(\widehat\tau)})^2}
\right.
\]
\[
\left.
 -\left(\frac{y_1(\widehat x)\widetilde y_2(\widehat x)}{\sqrt{1+\widetilde y_2^2(\widehat x)}}
-\frac{y_1(\widehat\tau)\widetilde y_2(\widehat
\tau)}{\sqrt{1+\widetilde y_2^2(\widehat\tau)}}\right)\frac{\sqrt{1+\widetilde y^2_2(\widehat
x)}-\sqrt{1+\widetilde y_2^2(\widehat\tau)}}{(\widehat x-\widehat
\tau)^2+(\sqrt{1+\widetilde y^2(\widehat x)}-\sqrt{1+\widetilde y_2^2(\widehat\tau)})^2} d\widehat
\tau \right|=o(1)
\]
when $\epsilon\to 0$ as follows from the lemma \ref{pust1} in Appendix. Indeed,
\[
\left|\left(\frac{y_1(\widehat x)y_2(\widehat
x)}{\sqrt{1+y_2^2(\widehat x)}}-\frac{y_1(\widehat x)\widetilde
y_2(\widehat x)}{\sqrt{1+\widetilde y_2^2(\widehat
x)}}\right)'\right|\lesssim\left\{
\begin{array}{cc}
\epsilon \widehat x, &\widehat x<1\\
\epsilon, &\widehat x>1
\end{array}
\right.
\]
\[
\left|\left(\sqrt{1+y_2^2(\widehat x)}-\sqrt{1+\widetilde
y_2^2(\widehat x)}\right)'\right|\lesssim\left\{
\begin{array}{cc}
\epsilon \widehat x, &\widehat x<1\\
\epsilon, &\widehat x>1
\end{array}
\right.
\]
and
\[
\left|\left(\frac{y_1(\widehat x)y_2(\widehat
x)}{\sqrt{1+y_2^2(\widehat x)}}\right)'\right|\lesssim
\frac{\widehat x}{\widehat x+1},\quad
\left|\left(\sqrt{1+y_2^2(\widehat x)}\right)'\right|\lesssim
\frac{\widehat x}{\widehat x+1}
\]
Notice also that, in the expression for $H$, the integral over every
finite interval gives the bounded contribution after division by $\sqrt{\widehat x^2+1}\log^+(\lambda^2(\widehat
x^2+1))$. We also have its stability in $y_2$.  Therefore, we can focus on \mbox{$\widehat \tau:
|\widehat x-\widehat \tau|>1$} only. We consider two cases: $\widehat x\in (0,1]$ and $\widehat x\in
[1,\lambda^{-1}]$.\smallskip

(1). Let $\widehat x\in (0,1]$. Clearly, we can assume that
$\widehat\tau\gg 1$.   Let
\[
H=\frac{B_1+B_2}{\sqrt{\widehat x^2+1}\log^+(\lambda^2(\widehat
x^2+1))}
\]
where
\[
B_1= \frac{y_1(\widehat x)y_2(\widehat x)}{\sqrt{1+y^2_2(\widehat
x)}}\int_0^{1/\lambda} \left(\frac{\sqrt{1+y_2^2(\widehat
x)}+\sqrt{1+y_2^2(\widehat \tau)}} {(\widehat
x+\widehat\tau)^2+(\sqrt{1+y_2^2(\widehat x)}+\sqrt{1+y_2^2(\widehat
\tau)})^2}-\right.
\]

\[
\left. \frac{\sqrt{1+y_2^2(\widehat x)}-\sqrt{1+y_2^2(\widehat
\tau)}} {(\widehat x-\widehat \tau)^2+(\sqrt{1+y_2^2(\widehat
x)}-\sqrt{1+y_2^2(\widehat \tau)})^2}\right)d\widehat \tau
\]
and
\[
B_2=\int_0^{1/\lambda} \frac{y_1(\widehat \tau)y_2(\widehat
\tau)}{\sqrt{1+y_2^2(\widehat\tau)}}
\left(\frac{\sqrt{1+y_2^2(\widehat x)}+\sqrt{1+y_2^2(\widehat
\tau)}} {(\widehat x+\widehat \tau)^2+(\sqrt{1+y_2^2(\widehat
x)}+\sqrt{1+y_2^2(\widehat \tau)})^2}+ \right.
\]

\[
\left. \frac{\sqrt{1+y_2^2(\widehat x)}-\sqrt{1+y_2^2(\widehat
\tau)}} {(\widehat x-\widehat \tau)^2+(\sqrt{1+y_2^2(\widehat
x)}-\sqrt{1+y_2^2(\widehat \tau)})^2}\right)d\widehat \tau
\]
We only need to handle integration over $\widehat \tau\in [2,
1/\lambda]$.

Consider $B_2$ first. The integrand has asymptotics
\[
y_1(\widehat\tau)\left(2\sqrt{1+y_2^2(\widehat x)}(\widehat
\tau^2+y_2^2(\widehat\tau))^{-1}-4y_2(\widehat \tau)\frac{\widehat
x\widehat\tau+\sqrt{1+y_2^2(\widehat\tau)}\sqrt{1+y_2^2(\widehat
x)}}{(\widehat\tau^2+y_2^2(\widehat\tau))^2} \right)(1+O(\widehat
\tau^{-1}))
\]
Thus, we immediately have a bound
\[
|B_2|\lesssim \log^+ \lambda
\]
Comparing the integral with the one where $y_2$ is replaced by
$\widetilde y_2$ gives us the necessary stability estimate
\begin{equation}\label{staab}
\left|\int_1^{1/\lambda}
\frac{y_1(\widehat\tau)}{\widehat\tau^2+y_2^2(\widehat\tau)}d\widehat\tau-\int_1^{1/\lambda}
\frac{y_1(\widehat\tau)}{\widehat\tau^2+\widetilde
y_2^2(\widehat\tau)}d\widehat\tau\right|=O(\epsilon)\log^+\lambda
\end{equation}
and the same estimates are valid for other integrals involved. For
the remainder $O(\widehat\tau^{-1})$, the corresponding function is
bounded by $C\widehat\tau^{-2}$ and this decay is integrable giving a uniformly small number when integrated over $[T,1/\lambda]$ with large $T$. For the integral over any finite interval $\widehat
\tau\in[0,T]$, the stability easily follows.  Thus, we
first take $T$ large and then send $\epsilon\to 0$. This will ensure the
stability in $y_2$.\bigskip

For $B_1$, the estimates are very similar. The estimate
\eqref{stab1} gives the stability for the first factor
 \[\frac{y_1(\widehat x)y_2(\widehat x)}{\sqrt{1+y^2_2(\widehat
x)}}\] and the asymptotics of the integrand is
$
\displaystyle \frac{2y_2(\widehat\tau)}{\widehat\tau^2+y_2^2(\widehat\tau)}+O(\widehat\tau^{-2})
$.
Thus, we can use an estimate similar to (\ref{staab}).

\bigskip

(2). Consider the case $\widehat x>1$ now and assume that $|\widehat
x-\widehat\tau|>1$ in the integration.

For $\widehat\tau>1$ and $\widehat x>1$, we can write
\[
\sqrt{1+y_2^2(\widehat x)}=y_2(\widehat x)(1+O(\widehat x^{-2}))
\]
\begin{equation}\label{truk}
\sqrt{1+y_2^2(\widehat
x)}+\sqrt{1+y_2^2(\widehat\tau)}=(y_2(\widehat
x)+y_2(\widehat\tau))R_1^{-1}
\end{equation}
\[
\sqrt{1+y_2^2(\widehat
x)}-\sqrt{1+y_2^2(\widehat\tau)}=(y_2(\widehat
x)-y_2(\widehat\tau))R_1
\]
and
\[
R_1=1+O\left(\frac{1}{\widehat x\widehat\tau}\right)
\]
Let us control how the integral will change if we replace
$\sqrt{1+y_2^2(\widehat x)}$ by $y_2(\widehat x)$ and
$\sqrt{1+y_2^2(\widehat x)}+\sqrt{1+y_2^2(\widehat\tau)}$ by
$y_2(\widehat x)+y_2(\widehat\tau)$. The errors produced in $B_2$,
for example, are at most
\[
C_1+C_2\int_1^{\lambda^{-1}}\left(\frac{1}{\widehat\tau}+\frac{1}{\widehat
x}\right)\frac{1}{|\widehat x-\widehat\tau|+1}d\widehat\tau\lesssim
1+\frac{\log\widehat x+\log^+ \lambda}{\widehat x}
\]
The estimate for $B_1$ is the same. Now, notice that
\[
\sup_{\widehat x>T, \lambda\in (0,1)}\frac{\log\widehat
x+\log^+ \lambda+\widehat x}{\widehat x^2\log^+(\lambda^2(\widehat
x^2+1))}\lesssim \frac{1}{\sqrt T}\to 0, \quad T\to \infty
\]
Since on every finite interval of integration $\widehat\tau\in
[0,T]$ we have stability in $y_2$, we  only need to handle
\[
\int_{\widehat\tau\in [0,\lambda^{-1}]} \left|\frac{(y_1(\widehat
x)+y_1(\widehat \tau))(y_2(\widehat x)+y_2(\widehat
\tau))}{(\widehat x+\widehat \tau)^2+(y_2(\widehat x)+y_2(\widehat
\tau))^2}- \frac{(y_1(\widehat x)-y_1(\widehat \tau))(y_2(\widehat
x)-y_2(\widehat \tau))}{(\widehat x-\widehat \tau)^2+(y_2(\widehat
x)-y_2(\widehat \tau))^2} \right|d\widehat \tau
\]
Let us change the variable  $\widehat\tau=\widehat x\alpha$ and
introduce two functions:
\begin{equation}\label{tuk}
f(\alpha, \widehat x)=\widehat x^{-1}y_1(\alpha\widehat x),\quad
g(\alpha,\widehat x)=\widehat x^{-1}y_2(\alpha\widehat x)
\end{equation}
As before, we have $f(0,\widehat x)=g(0,\widehat x)=0$,
\[
|\partial_\alpha f(\alpha,\widehat x)|=|y_1'(\alpha\widehat x)|\leq
1, \quad |f(\alpha,\widehat x)|\leq \alpha
\]
and
\[
|\partial_\alpha g(\alpha,\widehat x)-1|=|y_2'(\alpha\widehat
x)-1|\lesssim 1,
\]
Moreover, if $\widetilde g$ is the scaling of $\widetilde y_2$, then
\[
\|g'-\widetilde g'\|_{L^\infty[0,1/\lambda]}\leq \epsilon
\]
These estimates are uniform in $\widehat x$. The integral takes the
form
\[
\widehat
x\int_0^{1/x}\left|\frac{(f(1)+f(\alpha))(g(1)+g(\alpha))}{(
1+\alpha)^2+(g(1)+g(\alpha))^2}-\frac{(f(1)-f(\alpha))(g(1)-g(\alpha))}{(
1-\alpha)^2+(g(1)-g(\alpha))^2}\right|d\alpha
\]
We can rewrite
\[
\frac{(f(1)-f(\alpha))(g(1)-g(\alpha))}{(
1-\alpha)^2+(g(1)-g(\alpha))^2}=\frac{\displaystyle
\frac{f(1)-f(\alpha)}{1-\alpha}\cdot
\frac{g(1)-g(\alpha)}{1-\alpha}}{\displaystyle
1+\left(\frac{g(1)-g(\alpha)}{1-\alpha}\right)^2}
\]
and the lemma \ref{pust1} proves stability for the interval
$|\alpha-1|<1$. Then, the stability in $g$ can be easily seen for
every interval $\alpha\in [0,T]$ given fixed $T$ as the
corresponding error is $o(1)\widehat x$ when $\epsilon\to 0$ and
\[
o(1)\sup_{\widehat x>1} \frac{\widehat x}{\sqrt{\widehat
x^2+1}\log^+(\lambda^2(\widehat x^2+1))}=o(1)
\]
For large $\alpha$, we get the asymptotics
\[
\frac{(f(1)+f(\alpha))(g(1)+g(\alpha))}{(
1+\alpha)^2+(g(1)+g(\alpha))^2}-\frac{(f(1)-f(\alpha))(g(1)-g(\alpha))}{(
1-\alpha)^2+(g(1)-g(\alpha))^2}=
\]
\[
\frac{-4f(\alpha)g(\alpha)(\alpha+g(1)g(\alpha))}{(\alpha^2+g^2(\alpha))^2}+
\frac{2(f(1)g(\alpha)+g(1)f(\alpha)}{\alpha^2+g^2(\alpha)}+
O(\alpha^{-2})
\]
The error $O(\alpha^{-2})$ is integrable and the comparison of the
leading terms to the analogous expressions with $g$ replaced by
$\widetilde g$ gives the error at most
\[
o(1)\int_1^{1/x} \frac{d\alpha }{\alpha}=o(1) \log^+x
\]
This leads to the error of the size
\[
o(1)\frac{\widehat x \log^+ x}{\sqrt{\widehat
x^2+1}\log^+(x^2+\lambda^2)}=o(1), \quad \epsilon\to 0
\]
uniformly in $\lambda$ and $x>\lambda$.
\end{proof}
Now, we need to handle the other combination: $X_2+X_4$. The
analysis here is nearly identical and is based on the following
lemma.
\begin{lemma}Suppose $y_{1}, y_2, \widetilde y_2\in \dot Lip[0,\lambda^{-1}]$ and
\[
\|y'_1\|_{L^\infty[0,1/\lambda]}\leq 1,\quad \|y'_2-\widetilde y'_2\|_{L^\infty[0,1/\lambda]}\leq
\epsilon, \quad \|y'_2-1\|_{L^\infty[0,1/\lambda]}\ll 1
\]
If one defines
\[
H^{(1)}=\frac{1}{\sqrt{\widehat x^2+1}\log^+(\lambda^2(\widehat
x^2+1))}\int_0^{1/\lambda} \left(\left(\frac{y_1(\widehat
 x)y_2(\widehat x)}{\sqrt{1+y_2^2(\widehat x)}}+\frac{y_1(\widehat\tau)y_2(\widehat
\tau)}{\sqrt{1+y_2^2(\widehat\tau)}}\right)\right.
\]
\[
\times \frac{\sqrt{1+y_2^2(\widehat
x)}+\sqrt{1+y_2^2(\widehat\tau)}}{(\widehat x-\widehat
\tau)^2+(\sqrt{1+y_2^2(\widehat
x)}+\sqrt{1+y_2^2(\widehat\tau)})^2}
\]
\[
\left. -\left(\frac{y_1(\widehat x)y_2(\widehat
x)}{\sqrt{1+y_2^2(\widehat x)}}-\frac{y_1(\widehat\tau)y_2(\widehat
\tau)}{\sqrt{1+y_2^2(\widehat\tau)}}\right)\times\frac{\sqrt{1+y_2^2(\widehat
x)}-\sqrt{1+y_2^2(\widehat\tau)}}{(\widehat x+\widehat
\tau)^2+(\sqrt{1+y_2^2(\widehat x)}-\sqrt{1+y_2^2(\widehat\tau)})^2}
\right)d\widehat \tau
\]
and
\[
\widetilde H^{(1)}=\frac{1}{\sqrt{\widehat
x^2+1}\log^+(\lambda^2(\widehat x^2+1))}\int_0^{1/\lambda}\left(
\left(\frac{y_1(\widehat
 x)\widetilde y_2 (\widehat x)}{\sqrt{1+\widetilde y_2 ^2(\widehat x)}}+\frac{y_1(\widehat\tau)\widetilde y_2 (\widehat
\tau)}{\sqrt{1+\widetilde y_2 ^2(\widehat\tau)}}\right)\right.
\]
\[
 \times \frac{\sqrt{1+\widetilde y_2 ^2(\widehat
x)}+\sqrt{1+\widetilde y_2 ^2(\widehat\tau)}}{(\widehat x-\widehat
\tau)^2+(\sqrt{1+\widetilde y_2 ^2(\widehat x)}+\sqrt{1+\widetilde
y_2 ^2(\widehat\tau)})^2}
\]
\[
\left. -\left(\frac{y_1(\widehat x)\widetilde y_2 (\widehat
x)}{\sqrt{1+\widetilde y_2 ^2(\widehat
x)}}-\frac{y_1(\widehat\tau)\widetilde y_2 (\widehat
\tau)}{\sqrt{1+\widetilde y_2
^2(\widehat\tau)}}\right)\times \frac{\sqrt{1+\widetilde y_2 ^2(\widehat
x)}-\sqrt{1+\widetilde y_2 ^2(\widehat\tau)}}{(\widehat x+\widehat
\tau)^2+(\sqrt{1+\widetilde y_2 ^2(\widehat x)}-\sqrt{1+\widetilde
y_2 ^2(\widehat\tau)})^2} \right)d\widehat \tau
\]
then, uniformly in $y_{1}, y_2, \widetilde y_2$ and $\lambda\in
(0,1)$, we have
\[
\|H^{(1)}-\widetilde H^{(1)}\|_{L^\infty[0,1/\lambda]}=o(1), \quad
\epsilon\to 0
\]
and
\[
\|H^{(1)}\|_{L^\infty[0,1/\lambda]}\lesssim 1
\]
\end{lemma}
\begin{proof}The proof of this lemma repeats the argument for the previous
one word for word. The only minor change is contained in how we
handle the singularity in the denominator of $X_4$ when both $x$ and
$\tau$ go to zero. After the rescaling, we have an integral
\[
\left|\int_0^1\left(\frac{y_1(\widehat x)y_2(\widehat
x)}{\sqrt{1+y_2^2(\widehat x)}}-\frac{y_1(\widehat\tau)y_2(\widehat
\tau)}{\sqrt{1+y_2^2(\widehat\tau)}}\right)\frac{\sqrt{1+y_2^2(\widehat
x)}-\sqrt{1+y_2^2(\widehat\tau)}}{(\widehat x+\widehat
\tau)^2+(\sqrt{1+y_2^2(\widehat x)}-\sqrt{1+y_2^2(\widehat\tau)})^2}
d\widehat \tau\right|
\]
\[
\lesssim \int_0^1  \frac{|\widehat x-\widehat \tau|^2}{\widehat
x^2+\widehat\tau^2}d\widehat\tau\lesssim 1
\]
by the application of mean-value theorem. The stability of this
expression in $y_2$ follows from the lemma \ref{pust1}.
\end{proof}
This finishes the proof of theorem \ref{real-lemma}.
\end{proof}\bigskip
We continue now with the other terms: $P_4, P_5$ and $P_6$.\bigskip

{\bf (4).} Consider the term $P_4-P_4^0$. \smallskip

To study the stability in $t$, it is more convenient to rescale by
$\lambda$ and consider $y_1(\widehat x)=\lambda^{-1}f(\widehat x\lambda)$ and
$y_2(\widehat x)=\lambda^{-1}f_t(\widehat x\lambda)$. Then, the problem is
reduced to proving the stability of
\[
P_4=\frac{1}{\widehat x\sqrt{1+\widehat
x^2}\log^+(\lambda^2(\widehat x^2+1))}\frac{y_1(\widehat
x)y_2(\widehat x)}{\sqrt{1+y_2^2(\widehat x)}}\int_0^{1/\lambda}
\frac{y_2(\widehat
\tau)y_2'(\widehat\tau)}{\sqrt{1+y_2^2(\widehat\tau)}}K_2(\widehat
x,\widehat\tau,\sqrt{1+y_2^2(\widehat \tau}))d\widehat\tau
\]
in $y_2$. As before, we will be taking $\widetilde y_2$ with
$\|y_2'-\widetilde y_2'\|_{L^\infty[0,1/\lambda]}\leq \epsilon$ and making a
comparison.

By \eqref{stab1} and lemma \ref{auxi1}, we have
\[
\left|y_1(\widehat x)\left(\frac{y_2(\widehat
x)}{\sqrt{1+y_2^2(\widehat x)}}-\frac{\widetilde y_2(\widehat
x)}{\sqrt{1+\widetilde y^2(\widehat x)}}\right)\int_0^{1/\lambda}
\frac{y_2(\widehat
\tau)y_2'(\widehat\tau)}{\sqrt{1+y_2^2(\widehat\tau)}}K_2(\widehat
x,\widehat\tau,\sqrt{1+y_2^2(\widehat \tau}))d\widehat\tau\right|
\]
\[
\leq \epsilon \widehat x^3\log(1/\lambda), \quad \widehat x\in (0,1)
\]
and
\[
\leq \epsilon\log(1/x), \quad \widehat x>1
\]
Thus, after division, it gives an error at most $\epsilon$.

For the next term, \eqref{stab1} again gives
\[
\left|\frac{y_1(\widehat x)y_2(\widehat x)}{\sqrt{1+y_2^2(\widehat
x)}}\int_0^{1/\lambda} \left(\frac{y_2(\widehat
\tau)y_2'(\widehat\tau)}{\sqrt{1+y_2^2(\widehat\tau)}}-
\frac{\widetilde y_2(\widehat
\tau)\widetilde y_2'(\widehat\tau)}{\sqrt{1+\widetilde y_2^2(\widehat\tau)}}
\right)K_2(\widehat
x,\widehat\tau,\sqrt{1+y_2^2(\widehat \tau}))d\widehat\tau\right|
\]
\[
\lesssim \epsilon \widehat x\int_0^1 \widehat \tau |K_2(\widehat
x,\widehat\tau,\sqrt{1+y_2^2(\widehat\tau)})|d\widehat\tau+\epsilon\widehat
x\int_1^{1/\lambda}|K_2(\widehat
x,\widehat\tau,\sqrt{1+y_2^2(\widehat\tau)})|d\widehat\tau
\]
\[
\lesssim \epsilon \widehat x^2\log(1/\lambda), \quad \widehat x<1
\]
and
\[\lesssim \epsilon \widehat x^2\log(1/x), \quad \widehat x>1
\]
After division by $\widehat x \sqrt{\widehat x^2+1}\log^+(\lambda^2(\widehat x^2+1))$, it gives an error at most $O(\epsilon)$.

For the last term
\begin{equation}\label{staaa}
\frac{y_1(\widehat x)y_2(\widehat x)}{\sqrt{1+y_2^2(\widehat
x)}}\int_0^{1/\lambda} \frac{y_2(\widehat
\tau)y_2'(\widehat\tau)}{\sqrt{1+y_2^2(\widehat\tau)}}\Bigl(K_2(\widehat
x,\widehat\tau,\sqrt{1+y_2^2(\widehat \tau}))-K_2(\widehat
x,\widehat\tau,\sqrt{1+\widetilde y_2^2(\tau)})\Bigr) d\widehat\tau
\end{equation}
we can apply the mean value theorem and the resulting derivative of the kernel can be handled by the theorem \ref{real-lemma2} below.   As the result, the expression above can be bounded by
\[
\lesssim \widehat x^2\log(1/\lambda)o(1), \quad
\widehat x<1
\]
and
\[
\lesssim \widehat x^2\log(1/x)o(1),  \quad
\widehat x>1
\]
Upon division by
\[
\widehat x\sqrt{\widehat x^2+1}\log^+(\lambda^2(\widehat x^2+1))
\]
this is at most $o(1)$.\bigskip

{\bf (5).} The term $P_5-P_5^0$ can be estimated
similarly.\smallskip

Indeed, after scaling we have the following expression
\[
\sqrt{1+y_2^2(\widehat x)}\int_0^{1/\lambda}
\left(y_1'\frac{y_2}{\sqrt{1+y_2^2}}+
y_1y_2'(1+y_2^{2})^{-1.5}\right)K_2(\widehat
x,\widehat\tau,\sqrt{1+y_2^2(\widehat \tau}))d\widehat\tau
\]
and we can repeat the steps from the previous argument.\bigskip

{\bf (6).} We are left to handle $P_6-P_6^0$.\smallskip

This analysis is very similar to the one performed for $P_3$. However, we give details for completeness.

\begin{theorem}\label{real-lemma2}
\[
\left\|\frac{ Y_{\lambda,t}(L_1+\ldots+
L_4)-Y_{\lambda,0}(L_1^0+\ldots+
L_4^0)}{x\sqrt{x^2+\lambda^2}\log^+(x^2+\lambda^2)}\right\|_{L^\infty[0,1]}=o(1),
\quad t\to 0
\]
uniformly in $\lambda$.
\end{theorem}
\begin{proof}
Rescale by $\lambda$ and rewrite the problem for $y_1$ and $y_2$, as
before. Notice first that
\[
|\sqrt{1+y_2^2}-\sqrt{1+\widetilde y_2^2}|\leq  \epsilon \widehat x^2,
\quad \widehat x<1
\]
and
\[
|\sqrt{1+y_2^2}-\sqrt{1+\widetilde y_2^2}|\leq  \epsilon \widehat x,
\quad \widehat x>1
\]
so we only need to show that
\[
\left\|\frac{ (L_1+\ldots+ L_4)-(L_1^0+\ldots+ L_4^0)}{
x\log^+(x^2+\lambda^2)}\right\|_{L^\infty[0,1]}=o(1), \quad t\to 0
\]
and
\begin{equation}\label{ccc1}
\left\|\frac{L_1^0+\ldots+ L_4^0}{
x\log^+(x^2+\lambda^2)}\right\|_{L^\infty[0,1]}\lesssim 1
\end{equation}

We group $(L_1+L_2)-(L_1^0+L_2^0)$ and $(L_3+L_4)-(L_3^0+L_4^0)$ and
start with the following lemma which handles $L_3+L_4$.

\begin{lemma}\label{ddd1}
Suppose $y_{1}, y_{2}, \widetilde y_2 \in \dot Lip[0,\lambda^{-1}]$ and
\[
\|y'_1\|_{L^\infty[0,1/\lambda]}\leq 1,\quad \|y'_2-\widetilde y_2'\|_{L^\infty[0,1/\lambda]}\leq \epsilon, \quad \|y'_2-1\|_{L^\infty[0,1/\lambda]}\ll 1
\]
If one defines
\[
U=\frac{1}{\widehat x\log^+(\lambda^2(\widehat
x^2+1))}\int_0^{1/\lambda}
\left|\frac{y_2(\widehat\tau)y_2'(\widehat\tau)}{\sqrt{1+y_2^2(\widehat\tau)}}
\left[ \left(\frac{y_1(\widehat
 x)y_2(\widehat x)}{\sqrt{1+y_2^2(\widehat x)}}-\frac{y_1(\widehat\tau)y_2(\widehat
\tau)}{\sqrt{1+y_2^2(\widehat\tau)}}\right)\times\right.\right.
\]
\[
\left( \frac{\sqrt{1+y_2^2(\widehat
x)}-\sqrt{1+y_2^2(\widehat\tau)}}{(\widehat x+\widehat
\tau)^2+(\sqrt{1+y_2^2(\widehat
x)}-\sqrt{1+y_2^2(\widehat\tau)})^2} \left.\left.
-\frac{\sqrt{1+y_2^2(\widehat
x)}-\sqrt{1+y_2^2(\widehat\tau)}}{(\widehat x-\widehat
\tau)^2+(\sqrt{1+y_2^2(\widehat x)}-\sqrt{1+y_2^2(\widehat\tau)})^2}
\right)\right]\right.
\]
\[
-
\frac{
\widetilde y_2(\widehat\tau)\widetilde y_2'(\widehat\tau)}{\sqrt{1+\widetilde y_2^2(\widehat\tau)}}
\left[ \left(\frac{y_1(\widehat
 x)\widetilde y_2(\widehat x)}{\sqrt{1+\widetilde y_2^2(\widehat x)}}-\frac{y_1(\widehat\tau)\widetilde y_2(\widehat
\tau)}{\sqrt{1+\widetilde y_2^2(\widehat\tau)}}\right)\times\right.
\]
\[
\left( \frac{\sqrt{1+\widetilde y_2^2(\widehat
x)}-\sqrt{1+\widetilde y_2^2(\widehat\tau)}}{(\widehat x+\widehat
\tau)^2+(\sqrt{1+\widetilde y_2^2(\widehat
x)}-\sqrt{1+\widetilde y_2^2(\widehat\tau)})^2} \left.\left.
-\frac{\sqrt{1+\widetilde y_2^2(\widehat
x)}-\sqrt{1+\widetilde y_2^2(\widehat\tau)}}{(\widehat x-\widehat
\tau)^2+(\sqrt{1+\widetilde y_2^2(\widehat x)}-\sqrt{1+\widetilde y_2^2(\widehat\tau)})^2}
\right)\right]\right|d\widehat\tau
\]

then, uniformly in $y_{1}, y_2, \widetilde y_2$ and $\lambda\in (0,1)$, we have
\[
\|U\|_{L^\infty[0,1/\lambda]}=o(1), \quad \epsilon\to 0
\]
\end{lemma}
Notice that in this lemma we take an absolute value inside the
integration as that will make an argument more transparent.
\begin{proof} We first prove that
\[
\left\|\frac{1}{\widehat x\log^+(\lambda^2(\widehat
x^2+1))}\int_0^{1/\lambda}
\left|\left(\frac{y_2(\widehat\tau)y_2'(\widehat\tau)}{\sqrt{1+y_2^2(\widehat\tau)}}-
\frac{\widetilde y_2(\widehat\tau)\widetilde y_2'(\widehat\tau)}{\sqrt{1+\widetilde y_2^2(\widehat\tau)}}\right)
\right.\right.\times
\]
\begin{equation}\label{prom1}
 \left(\frac{y_1(\widehat
 x)y_2(\widehat x)}{\sqrt{1+y_2^2(\widehat x)}}-\frac{y_1(\widehat\tau)y_2(\widehat
\tau)}{\sqrt{1+y_2^2(\widehat\tau)}}\right)\times \left(
\frac{\sqrt{1+y_2^2(\widehat
x)}-\sqrt{1+y_2^2(\widehat\tau)}}{(\widehat x+\widehat
\tau)^2+(\sqrt{1+y_2^2(\widehat
x)}-\sqrt{1+y_2^2(\widehat\tau)})^2}\right.
\end{equation}
\[
\left. \left.\left. -\frac{\sqrt{1+y_2^2(\widehat
x)}-\sqrt{1+y_2^2(\widehat\tau)}}{(\widehat x-\widehat
\tau)^2+(\sqrt{1+y_2^2(\widehat x)}-\sqrt{1+y_2^2(\widehat\tau)})^2}
\right)  \right| d\widehat
\tau\right\|_{L^\infty[0,\lambda^{-1}]}=o(1)
\]
as $\epsilon\to 0$, uniformly in parameters. Let us
observe that
\[
\left|\frac{y_2(\widehat\tau)y_2'(\widehat\tau)}{\sqrt{1+y_2^2(\widehat\tau)}}-
\frac{\widetilde y_2(\widehat\tau)\widetilde y_2'(\widehat\tau)}{\sqrt{1+\widetilde y_2^2(\widehat\tau)}}\right|
\lesssim \epsilon\widehat \tau, \quad \widehat\tau <1
\]
and
\[
\left|\frac{y_2(\widehat\tau)y_2'(\widehat\tau)}{\sqrt{1+y_2^2(\widehat\tau)}}-
\frac{\widetilde y_2(\widehat\tau)\widetilde y_2'(\widehat\tau)}{\sqrt{1+\widetilde y_2^2(\widehat\tau)}}\right|
\lesssim \epsilon, \quad \widehat\tau >1
\]
Therefore, to show (\ref{prom1}) it is sufficient to use an estimate
\eqref{chetyre} proved below,  and the following inequality
\[
\frac{1}{\widehat x\log^+(\lambda^2(\widehat
x^2+1))}\int_0^{1/\lambda}
\left|\frac{\widehat\tau}{\sqrt{\widehat\tau^2+1}}\times\right.
\]
\begin{equation}\label{prom22}
 \left(\frac{y_1(\widehat
 x)y_2(\widehat x)}{\sqrt{1+y_2^2(\widehat x)}}-\frac{y_1(\widehat\tau)y_2(\widehat
\tau)}{\sqrt{1+y_2^2(\widehat\tau)}}\right)\times \left(
\frac{\sqrt{1+y_2^2(\widehat
x)}-\sqrt{1+y_2^2(\widehat\tau)}}{(\widehat x+\widehat
\tau)^2+(\sqrt{1+y_2^2(\widehat
x)}-\sqrt{1+y_2^2(\widehat\tau)})^2}\right.
\end{equation}
\[
\left.\left. -\frac{\sqrt{1+y_2^2(\widehat
x)}-\sqrt{1+y_2^2(\widehat\tau)}}{(\widehat x-\widehat
\tau)^2+(\sqrt{1+y_2^2(\widehat x)}-\sqrt{1+y_2^2(\widehat\tau)})^2}
\right)  \right| d\widehat \tau\lesssim 1
\]
The latter can be achieved in a standard way by following, e.g, the estimates in the proof of \eqref{chetyre}.

Now, consider
\[
U^{(1)}=\frac{1}{\widehat x\log^+(\lambda^2(\widehat
x^2+1))}\int_0^{1/\lambda}
\frac{\widehat\tau}{\sqrt{1+\widehat\tau^2}} |F(\widehat x,\widehat
\tau)-F_0(\widehat x,\widehat \tau)|d\widehat \tau
\]
where
\[
F=
 \left(\frac{y_1(\widehat
 x)y_2(\widehat x)}{\sqrt{1+y_2^2(\widehat x)}}-\frac{y_1(\widehat\tau)y_2(\widehat
\tau)}{\sqrt{1+y_2^2(\widehat\tau)}}\right)\left(
\frac{\sqrt{1+y_2^2(\widehat
x)}-\sqrt{1+y_2^2(\widehat\tau)}}{(\widehat x+\widehat
\tau)^2+(\sqrt{1+y_2^2(\widehat
x)}-\sqrt{1+y_2^2(\widehat\tau)})^2}\right.
\]
\[
\left. -\frac{\sqrt{1+y_2^2(\widehat
x)}-\sqrt{1+y_2^2(\widehat\tau)}}{(\widehat x-\widehat
\tau)^2+(\sqrt{1+y_2^2(\widehat x)}-\sqrt{1+y_2^2(\widehat\tau)})^2}
\right)
\]
and
\[
F_0=
 \left(\frac{y_1(\widehat
 x)\widetilde y_2(\widehat x)}{\sqrt{1+\widetilde y_2^2(\widehat x)}}-\frac{y_1(\widehat\tau)\widetilde y_2(\widehat
\tau)}{\sqrt{1+\widetilde y_2^2(\widehat\tau)}}\right)\left(
\frac{\sqrt{1+\widetilde y_2^2(\widehat
x)}-\sqrt{1+\widetilde y_2^2(\widehat\tau)}}{(\widehat x+\widehat
\tau)^2+(\sqrt{1+\widetilde y_2^2(\widehat
x)}-\sqrt{1+\widetilde y_2^2(\widehat\tau)})^2}\right.
\]
\[
\left. -\frac{\sqrt{1+\widetilde y_2^2(\widehat
x)}-\sqrt{1+\widetilde y_2^2(\widehat\tau)}}{(\widehat x-\widehat
\tau)^2+(\sqrt{1+\widetilde y_2^2(\widehat x)}-\sqrt{1+\widetilde y_2^2(\widehat\tau)})^2}
\right)
\]
We are going to prove that
\begin{equation}\label{chetyre}
\|U^{(1)}\|_{L^\infty[0,1/\lambda]}=o(1), \quad \epsilon\to 0
\end{equation}

 Consider the case $\widehat x\in
[0,1]$. The regime  $\widehat x\to 0$ is what makes the difference
when compared to the same analysis for $P_3$. Take $F$ and rewrite it as follows
\[
F= -4\widehat x\widehat \tau
 \left(\frac{y_1(\widehat
 x)y_2(\widehat x)}{\sqrt{1+y_2^2(\widehat x)}}-\frac{y_1(\widehat\tau)y_2(\widehat
\tau)}{\sqrt{1+y_2^2(\widehat\tau)}}\right)\times
\frac{\sqrt{1+y_2^2(\widehat
x)}-\sqrt{1+y_2^2(\widehat\tau)}}{(\widehat x+\widehat
\tau)^2+(\sqrt{1+y_2^2(\widehat
x)}-\sqrt{1+y_2^2(\widehat\tau)})^2}
\]
\[
 \times\frac{1}{(\widehat x-\widehat
\tau)^2+(\sqrt{1+y_2^2(\widehat x)}-\sqrt{1+y_2^2(\widehat\tau)})^2}
\]
The lemma \ref{pust1} yields
\[
\int_\sigma^T |F-F_0|d\widehat \tau
=\widehat xo(1), \quad \epsilon\to 0
\]
for every fixed $T>\sigma>0$. For the integration over $[0,\sigma]$, we get
\[
\int_0^\sigma (|F|+|F_0|)d\widehat \tau \lesssim  \widehat x\int_0^\sigma \frac{(\widehat x+\widehat\tau)\widehat\tau}{\widehat x^2+\widehat \tau^2}d\widehat\tau\lesssim \widehat x\sigma
\]
This gives
\[
\int_0^T |F-F_0|d\widehat \tau =\widehat x
o(1), \quad \epsilon\to 0
\]
Now, for $\widehat x\in [0,1]$, the asymptotics for large $\widehat\tau$ are
\[
F=-\frac{4\widehat x \widehat\tau  y_1(\widehat\tau)y_2(\widehat\tau)}{(\widehat\tau^2+y_2^2(\widehat\tau))^2}+O(\widehat\tau^{-2}), \quad
F_0=-\frac{4\widehat x \widehat\tau  y_1(\widehat\tau)\widetilde y_2(\widehat\tau)}{(\widehat\tau^2+\widetilde y_2^2(\widehat\tau))^2}+O(\widehat\tau^{-2})
\]
and therefore
\[
\sup_{\widehat x\in [0,1]}\int_T^{1/\lambda}|F-F_0|d\widehat\tau=o(1)\widehat x \log^+ \lambda+CT^{-1}
\]
That shows $U^{(1)}$ is small uniformly in $\lambda$ and $\widehat x\in [0,1]$ as long as $\epsilon\to 0$.
Similarly, we can handle an interval $\widehat x\in [0,T]$
 with arbitrary large fixed $T$. In case of $\widehat x>T$, we can treat the interval $|\widehat \tau-\widehat x|<1$ using lemma \ref{pust1} as before. Outside this interval, we again use \eqref{truk} to get (compare with \eqref{tuk})
\[
\int_1^{1/\lambda}|F-F_0|d\widehat\tau \lesssim
\]
\[
 \widehat x \int_0^{1/x}u|f(1)-f(u)|\left|\frac{g(1)-g(u)}{((1+u)^2+(g(1)-g(u))^2)((1-u)^2+(g(1)-g(u))^2)} \right.
\]
\[\left.
-\frac{\widetilde g(1)-\widetilde g(u)}{((1+u)^2+(\widetilde g(1)-\widetilde g(u))^2)((1-u)^2+(\widetilde g(1)-\widetilde g(u))^2)}
  \right|du
+\log^+\lambda
\]
Computing the asymptotics at infinity, we obtain that the last quantity is
\[
o(1)\widehat x\log^+ x, \quad \epsilon\to 0
\]
Then,
\[
\sup_{\widehat x>T} \frac{o(1)\widehat x\log^+ x+\log^+\lambda}{\widehat x\log^+(\lambda^2(\widehat x^2+1))}=o(1)+T^{-1/2}
\]
as long as $T<\lambda^{-1/2}$. This bound proves that $U^{(1)}$ is small.

\end{proof}

The combination $L_1+L_2$ is handled similarly. We need the
following lemma for that.
\begin{lemma}\label{ddd2}
Suppose $y_{1}, y_{2}, \widetilde y_2 \in \dot Lip[0,\lambda^{-1}]$ and
\[
\|y'_1\|_{L^\infty[0,1/\lambda]}\leq 1,\quad \|y'_2-\widetilde y_2'\|_{L^\infty[0,1/\lambda]}\leq \epsilon, \quad \|y'_2-1\|_{L^\infty[0,1/\lambda]}\ll 1
\]
If one defines
\[
V=\frac{1}{\widehat x\log^+(\lambda^2(\widehat
x^2+1))}\int_0^{1/\lambda}
\left|\frac{y_2(\widehat\tau)y_2'(\widehat\tau)}{\sqrt{1+y_2^2(\widehat\tau)}}
\left[ \left(\frac{y_1(\widehat
 x)y_2(\widehat x)}{\sqrt{1+y_2^2(\widehat x)}}+\frac{y_1(\widehat\tau)y_2(\widehat
\tau)}{\sqrt{1+y_2^2(\widehat\tau)}}\right)\times\right.\right.
\]
\[
\left( \frac{\sqrt{1+y_2^2(\widehat
x)}+\sqrt{1+y_2^2(\widehat\tau)}}{(\widehat x+\widehat
\tau)^2+(\sqrt{1+y_2^2(\widehat
x)}+\sqrt{1+y_2^2(\widehat\tau)})^2} \left.\left.
-\frac{\sqrt{1+y_2^2(\widehat
x)}+\sqrt{1+y_2^2(\widehat\tau)}}{(\widehat x-\widehat
\tau)^2+(\sqrt{1+y_2^2(\widehat x)}+\sqrt{1+y_2^2(\widehat\tau)})^2}
\right)\right]\right.
\]
\[
-
\frac{
\widetilde y_2(\widehat\tau)\widetilde y_2'(\widehat\tau)}{\sqrt{1+\widetilde y_2^2(\widehat\tau)}}
\left[ \left(\frac{y_1(\widehat
 x)\widetilde y_2(\widehat x)}{\sqrt{1+\widetilde y_2^2(\widehat x)}}+\frac{y_1(\widehat\tau)\widetilde y_2(\widehat
\tau)}{\sqrt{1+\widetilde y_2^2(\widehat\tau)}}\right)\times\right.
\]
\[
\left( \frac{\sqrt{1+\widetilde y_2^2(\widehat
x)}+\sqrt{1+\widetilde y_2^2(\widehat\tau)}}{(\widehat x+\widehat
\tau)^2+(\sqrt{1+\widetilde y_2^2(\widehat
x)}+\sqrt{1+\widetilde y_2^2(\widehat\tau)})^2} \left.\left.
-\frac{\sqrt{1+\widetilde y_2^2(\widehat
x)}+\sqrt{1+\widetilde y_2^2(\widehat\tau)}}{(\widehat x-\widehat
\tau)^2+(\sqrt{1+\widetilde y_2^2(\widehat x)}+\sqrt{1+\widetilde y_2^2(\widehat\tau)})^2}
\right)\right]\right|d\widehat\tau
\]
then, uniformly in $y_{1}, y_2, \widetilde y_2$ and $\lambda\in (0,1)$, we have
\[
\|V\|_{L^\infty[0,1/\lambda]}=o(1), \quad \epsilon\to 0
\]
\end{lemma}

\begin{proof}The proof of this lemma is nearly identical. It is actually easier as the singularities in the denominator are absent.\end{proof}

The bound \eqref{ccc1} follows easily from the arguments given
in the proofs of lemmas \ref{ddd1} and \ref{ddd2}.
The proof of the theorem \ref{real-lemma2} is now finished.
\end{proof}
\bigskip
\subsection{The bound \eqref{dwa}}
The estimate \eqref{dwa} was in fact already proved in the previous
subsection. Indeed, recall \eqref{proizvol}. The derivative of $F$
involves six terms: $I_1+\ldots+I_6$.

For instance, $I_2$ gives the following operator
\[
\left(\frac{1}{x\sqrt{x^2+\lambda^2}\log^+(x^2+\lambda^2)}f\int_0^1 K_1(x,\tau,\sqrt{\lambda^2+f^2})d\tau\right)v'
\]
from $\dot Lip[0,1]$ to $L^\infty[0,1]$. Take $f=x+u$ where $\|u\|_{\dot
Lip[0,1]}\leq \epsilon$. Then, one needs to show that
\[
\sup_{\lambda\in (0,1], \|v\|_{\dot Lip[0,1]}\leq 1, \|u\|_{\dot Lip[0,1]}\leq  \epsilon}\left\|
\frac{1}{x\sqrt{x^2+\lambda^2}\log^+(x^2+\lambda^2)}\left(f\int_0^1 K_1(x,\tau,\sqrt{\lambda^2+f^2})d\tau\right.\right.
\]
\[
\left.\left.
-x\int_0^1 K_1(x,\tau,\sqrt{\lambda^2+x^2})d\tau\right)
 v'\right\|_{L^\infty[0,1]}=o(1), \quad \epsilon\to 0
\]
The proof of that, however, repeats the one for \eqref{snoska1} where
$\rho=x$. All other terms corresponding to $\{I_j\}_{j\neq 2}$ can
be handled similarly and that gives \eqref{dwa}.\bigskip

\section{The proof of the main theorem and regularity of solutions.}

We start with proving theorem \ref{main}.

\begin{proof}
 We can rewrite the equation \eqref{operat} as
\[
\psi={\cal O}\psi
\]
and the items (a), (b), and (c) stated on the same page were all justified. In particular,
we can choose sufficiently small $\delta$ and $\lambda_0$ such that
for every $\lambda\in (0,\lambda_0)$ the operator $\cal{O}$ has the unique fixed point in
 $\cal{B}_\delta=\{\psi: \|\psi\|_{\dot Lip[0,1]}\leq \delta\}$. It follows from the construction (and \eqref{why-small} in particular) that the solution
\[
y(x,\lambda)=\sqrt{\lambda^2+(x+\psi(x,\lambda))^2}
\]
converges to $|x|$ as $\lambda\to 0$. Moreover, one immediately has $y(x,\lambda)\in Lip[-1,1]$. Since $y$ is positive, one can substitute it to the equation and get $y\in C^1[-1,1]$. This regularity, however, will be significantly improved in the next theorem.
\end{proof}

{\bf Remark.} The self-similar behavior around the origin predicted by \eqref{predict} is an immediate corollary of \eqref{why-small}. \smallskip

Let us prove now that the solution $y(x,\lambda)$ is actually infinitely smooth.
\begin{theorem}
For every $\lambda\in (0,\lambda_0)$, we have $y(x,\lambda)\in C^\infty(-1,1)$.
\end{theorem}
\begin{proof}

The bound (\ref{vza}) implies that $K_1(x,\xi,y)>0$ and thus $\int_{-1}^1 K(x,\xi,y)d\xi>0$ as well. We have
\begin{equation}\label{zhaba}
y'(x,\lambda)=\frac{\displaystyle \int_{-1}^1 y'(\xi,\lambda) K(x,\xi,y)d\xi}{\displaystyle \int_{-1}^1 K(x,\xi,y)d\xi}
\end{equation}
and one might want to differentiate this expression consecutively hoping to use the standard bootstrapping argument.
Recall that 
\[
K(x,\xi,y)=\log\left((x+\xi)^2+(y(x)+y(\xi))^2\right)-\log\left((x-\xi)^2+(y(x)-y(\xi))^2\right)
\]
and the first term presents no problem for bootstrapping as $\log$ is smooth on $(0,\infty)$ and  $(x+\xi)^2+(y(x)+y(\xi))^2$ is strictly positive. However, the second term
$\log\left((x-\xi)^2+(y(x)-y(\xi))^2\right)$ might be problematic. We will show now how to handle it. Notice that
all potentially singular integrals in \eqref{zhaba} can be written as
\begin{equation}\label{bad-term}
\int_{-1}^1 g(\xi)\log((x-\xi)^2+(y(x)-y(\xi)^2)d\xi
\end{equation}
where $g$ is either equal to $1$ or to $y'(\xi)$. The logarithm can be represented as 
\[
\log((x-\xi)^2+(y(x)-y(\xi)^2)=2\log|x-\xi|+\log\left(1+\left(\frac{y(x)-y(\xi}{x-\xi}\right)^2\right)
\]

 Suppose we fix $\lambda$ so small that the contraction mapping works. We take
$
H_\delta(x)=\log(\sqrt{\delta^2+x^2})
$ instead of $H(x)=\log x$ and denote the corresponding kernel by $K_\delta$. Then, in a similar way, one can prove the existence of $y_\delta(x,\lambda)$ and $y_\delta(x,\lambda)\to y(x,\lambda), \delta\to 0$ uniformly over $[-1,1]$. Since  $H_\delta\in C^\infty(-1,1)$, we immediately get $y_\delta(x,\lambda)\in C^\infty(-1,1)$ so the lemmas from the Appendix are applicable.
We want to obtain estimates on $\|y_\delta\|_{C^{n}[-a,a]}$ that are  uniform in $\delta$.

To this end,  proceed by induction. Our inductive assumption is that $\|y^{(n)}_\delta\|_{L^\infty[-b,b]}<C(n,b)$ with every $b:
b<1$, uniformly in $\delta$. The contraction mapping argument gives us this condition for $n=1$.
Now, let us show how to use the lemmas from the Appendix to cover $n=2$. We set $\epsilon=1/2$.

Consider
\begin{equation}
y_\delta'(x) P(x)=\int_{-1}^1 y'_\delta(\xi)K_\delta(x,\xi,y_\delta)d\xi \label{nyan}
\end{equation}
with
\[
\quad P(x)=\int_{-1}^1 K_\delta(x,\xi,y_\delta)d\xi
\]
Then,
\[
\Delta_{x_1,x_2}(y_\delta'P)=(\Delta_{x_1,x_2} y_\delta')P(x_1)+(\Delta_{x_1,x_2}P)y_\delta'(x_2)
\]
and so
\[
(\Delta_{x_1,x_2} y_\delta')P(x_1)=-(\Delta_{x_1,x_2}P)y_\delta'(x_2)+\Delta_{x_1,x_2}\left(\int_{-1}^1 K_\delta(x,\xi,y_\delta)d\xi\right)
\]
The first step is to show that $\|y_\delta'\|_{C^{1/2}[-b,b]}$ is bounded uniformly in $\delta$ for every $b<1$. To this end, it is sufficient to estimate $\Delta_{x_1,x_2}y'_\delta$.
Notice that $P$ is positive and so poses no problem. The factor $y_\delta'$ is uniformly bounded by the inductive assumption. Consider 
\begin{equation}\label{crest}
\Delta_{x_1,x_2}P,\quad\Delta_{x_1,x_2}\left(\int_{-1}^1 K_\delta(x,\xi,y_\delta)d\xi\right)
\end{equation}
and focus on the terms of the form \eqref{bad-term}. In $P$, the function 
\[
\int_{-1}^1 \log|x-\xi|d\xi 
\]
is smooth. For 
\[
\int_{-1}^1 \log\left(1+\left(\frac{y_\delta(x)-y_\delta(\xi}{x-\xi}\right)^2\right)d\xi
\]
we apply lemma \ref{difff} and an interpolation bound
\[
\sup_{x_1,x_2\in [-b,b]}\left|\frac{\Delta_{x_1,x_2}f}{|x_1-x_2|^{1/2}}\right|\lesssim \sqrt{\|f\|_{C^1[-b,b]}\|f\|_{C[-b,b]}}
\]
to get
\[
\sup_{x_1,x_2\in [-b,b]}\left|\frac{\Delta_{x_1,x_2}P}{|x_1-x_2|^{1/2}}\right|\lesssim \left(\|y_\delta\|_{C^{1.5}[-b,b]}\right)^{1/2}
\]
For the second function in \eqref{crest}, we argue similarly. The estimate \eqref{zanud} gives
\[
\sup_{x_1,x_2\in [-b,b]}\left|\frac{\Delta_{x_1,x_2}
\displaystyle \int_{-1}^1 y'_\delta(\xi)\log|x-\xi|d\xi
}{|x_1-x_2|^{1/2}}\right|\lesssim 1
\] 
by the inductive assumption. Therefore, we get
\[
\|y_\delta'\|_{C^{1/2}[-b,b]}\lesssim 1+\left(\|y_\delta'\|_{C^{1/2}[-b,b]}\right)^{1/2}
\]
which implies the uniform bound on $\|y_\delta'\|_{C^{1/2}[-b,b]}$ for any $b<1$.

Now, differentiate \eqref{nyan} to get
\[
y_\delta'' P+y_\delta' P'=\left(\int_{-1}^1 y'_\delta(\xi)K_\delta(x,\xi,y_\delta)d\xi\right)'
\]
We have $P'\in C[-b,b]$ by lemma \ref{difff}. Then, 
\[
\left(\int_{-1}^1 y'_\delta(\xi)K_\delta(x,\xi,y_\delta)d\xi\right)'\in C[-b,b]
\]
with bounds uniform in $\delta$ as follows from lemma \ref{difff} and \eqref{krot}. 
 That shows $\|y_\delta\|_{C^n[-b,b]}$ is bounded uniformly in $\delta$ for $n=2$.
 
For a general $n$, we argue similarly.
Differentiation \eqref{nyan} $(n-1)$ times gives 
 \[
y_\delta^{(n)}(x) P(x)+\Omega_{n-1}(x)=\partial^{(n-1)}_x\int_{-1}^1 y'_\delta(\xi)K_\delta(x,\xi,y_\delta)d\xi
\]
Using the inductive assumption, we first show that all norms $\|y_\delta\|_{C^{n+0.5}[-b,b]}$ are bounded uniformly in $\delta$. Then, we bootstrap that to $C^{n+1}$ norm.

Once the $\delta$-independent bounds for  $\|y_\delta\|_{C^{(n)}[-a,a]}$ are established, we can take $\delta\to 0$.
That gives $y(x,\lambda)\in C^{n}[-a,a]$ for every $n$. Indeed, there is a sequence $\{y_{\delta_j}\}\to u$ in $C^n[-a,a]$ by Arzela-Ascoli and so $u\in C^n[-a,a]$. However this includes the uniform convergence so $y=u$. Since $n$ is arbitrary, we get the statement of the theorem.
\end{proof}\smallskip

\section{Appendix.}
\begin{lemma}\label{pust1}
If $\|f'-g'\|_{L^\infty[0,T]}\leq \delta$, then
\[
\left|\frac{f(x)-f(y)}{x-y}-\frac{g(x)-g(y)}{x-y}\right|\leq \delta
\]
uniformly in $x,y\in [0,T]$.
\end{lemma}
\begin{proof}
Indeed, it follows from the following representation
\begin{equation}\label{cros}
\Upsilon_f(x,y)=\frac{f(x)-f(y)}{x-y}=\int_0^1 f'(y+(x-y)t)dt
\end{equation}
\end{proof}

The next lemmas are needed to show that the solution $y(x,\lambda)$ is infinitely smooth.
\begin{lemma}
\label{loga}
Suppose $f\in C^{\infty}(-1,1)$ and $0<a<b\leq 1$. Then, for every $\epsilon\in (0,1)$,
\begin{equation}\label{krot}
\left\|\int_{-1}^1 f(\xi) \log|x-\xi|d\xi               \right\|_{C^{n}[-a,a]}<C(n,a,b,\epsilon)(\|f\|_{C^{n-1+\epsilon}[-b,b]}+\|f\|_{L^\infty[-1,1]})
\end{equation}
and
\begin{equation}\label{zanud}
\left\|\int_{-1}^1 f(\xi) \log|x-\xi|d\xi               \right\|_{C^{n+\epsilon}[-a,a]}<C(n,a,b,\epsilon)(\|f^{(n)}\|_{L^\infty[-b,b]}+\|f\|_{L^\infty[-1,1]})
 \end{equation}
\end{lemma}
\begin{proof}
The convolution structure of the kernel implies that it is sufficient to prove the statement for $n=1$ only. This amounts to checking that
\[
\left\|\int_{-1}^1 \frac{f(x)-f(\xi)}{x-\xi}d\xi\right\|_{C[-a,a]}\lesssim \|f\|_{C^{\epsilon}[-b,b]}+\|f\|_{L^\infty[-1,1]}
\]
which is trivial. The estimate \eqref{zanud} can be obtained in a similar way.
\end{proof}

\begin{lemma}\label{difff}
Suppose $f(x)\in C^\infty[-1,1]$ and $g(x)\in C[-1,1]$. Then
\[
\left\|\int_{-1}^1g(\xi)\log\left(1+\left(\frac{f(x)-f(\xi)}{x-\xi}\right)^2\right)d\xi\right\|_{C^1[-1,1]}<C_\epsilon\|f\|_{C^{1+\epsilon}[-1,1]}\|g\|_{C[-1,1]}
\]
with $C_\epsilon$ independent of $f$.
\end{lemma}
\begin{proof}
We  write (\ref{cros}) and differentiate to get
\[
\left|\int_{-1}^1 g(\xi) \frac{2\Upsilon_f(x,\xi)}{1+\Upsilon_f^2(x,\xi)}\left(\int_0^1f''(\xi+(x-\xi)t)tdt \right)d\xi\right|
\]
\[
\lesssim \|g\|_{C[-1,1]}\int_{-1}^1\left|\int_0^1 \frac{\partial_t(f'(\xi+(x-\xi)t)-f'(\xi))}{x-\xi}tdt\right|d\xi
\]
\[
\lesssim \|g\|_{C[-1,1]}\int_{-1}^1 \frac{\|f'\|_{C^\epsilon[-1,1]}|x-\xi|^\epsilon}{|x-\xi|}d\xi\lesssim \epsilon^{-1}\|f\|_{C^{1+\epsilon}[-1,1]}\|g\|_{C[-1,1]}
\]
\end{proof}
By consecutive differentiation, one gets
\begin{lemma}\label{joke}
Suppose $f(x)\in C^\infty[-1,1]$ and $g(x)\in C[-1,1]$. Then
\[
\left\|\int_{-1}^1g(\xi)\log\left(1+\left(\frac{f(x)-f(\xi)}{x-\xi}\right)^2\right)
d\xi\right\|_{C^n[-1,1]}<\Bigl(C_n(\epsilon)
 \|f\|_{C^{n+\epsilon}[-1,1]}+F_n(\|f\|_{C^n[-1,1]})\Bigr)\|g\|_{C[-1,1]}
\]
where $F_n$ is a certain function of $\|f\|_{C^{n}[-1,1]}$ only.
\end{lemma}
\begin{proof}
The proof is identical to the previous one.
\end{proof}
{\bf Remark.} The lemmas \ref{loga} and \ref{joke} will hold true if we replace $\log x$ by $\log \sqrt{x^2+\delta^2}$. The resulting estimates will be $\delta$ independent.

\section{Acknowledgment.}
This research was supported by NSF grants DMS-1067413 and FRG
DMS-1159133. The author thanks Kyudong Choi who made some suggestions that greatly improved the exposition of the paper.

\end{document}